\documentclass[11pt,reqno]{amsart}%,oneside

\usepackage{setspace}
\usepackage{hyperref}

\usepackage{color}
\usepackage[all]{xy}
\usepackage{amssymb, amsmath, amsfonts, amsthm}
\usepackage{enumitem}
\usepackage{comment}

\calclayout

\usepackage{amsmath,amssymb}
\usepackage{amsthm}
\usepackage{mathabx}
\usepackage{amsfonts}
\usepackage{pst-node}
\usepackage{graphicx}
\usepackage{pifont}
\usepackage[T1]{fontenc}
\usepackage{tikz-cd}
 
\newtheorem{theorem}{Theorem}[section]
\newtheorem{proposition}[theorem]{Proposition}
\newtheorem{corollary}[theorem]{Corollary}
\newtheorem{lemma}[theorem]{Lemma}
\theoremstyle{definition}
\newtheorem{definition}[theorem]{Definition}
\newtheorem{example}[theorem]{Example}

\newtheorem{remark}[theorem]{Remark}

\makeatletter
\@addtoreset{equation}{section}
\makeatother

\newcommand{\bbZ}{{\mathbb Z}} % the set of integer numbers
\newcommand{\bbF}{{\mathbb F}} % the unit circle (the one dimensional torus)

\newcommand{\bC}{{\mathbf C}}

\newcommand{\cC}{{\mathcal C}}
\newcommand{\cCe}{{\mathcal Ce}}

\newcommand{\cE}{{\mathcal E}}
\newcommand{\cF}{{\mathcal F}}
\newcommand{\cG}{{\mathcal G}}
\newcommand{\cH}{{\mathcal H}}

\newcommand{\cL}{{\mathcal L}}
\newcommand{\cM}{{\mathcal M}}
\newcommand{\cN}{{\mathcal N}}

\newcommand{\cO}{{\mathcal O}}

\newcommand{\id}{\mathrm{id}}

\DeclareMathOperator*{\hocolim}{hocolim}

\def\maprt#1{\smash{\,\mathop{\longrightarrow}\limits^{#1}\,}}

\newcommand{\Aut}{{\rm Aut}}
\newcommand{\Inn}{{\rm Inn}}

\newcommand{\Hom}{\mathrm{Hom}}

\newcommand{\Mor}{\mathrm{Mor}}
\newcommand{\Map}{\mathrm{Map}}
\newcommand{\Ext}{\mathrm{Ext}}

\newcommand{\Res}{\mathrm{Res}}

\newcommand{\bbZp}{\mathbb{Z}_{(p)}}

\newcommand{\Out}{\mathrm{Out}}

\begin{document}
\title{Higher limits over the fusion orbit category via subgroup complexes}
 
\author{Erg\"un Yal{\c c}{\i}n}
\address{Bilkent University, Department of Mathematics, 06800, Bilkent, Ankara, Turkey}

\email{yalcine@fen.bilkent.edu.tr}

\begin{abstract}  
Let $\mathcal{F}$ be a saturated fusion system over a finite $p$-group $S$, and $\mathcal{O}^c (\mathcal{F})$ denote the centric orbit category of $\mathcal{F}$. Suppose that $G$ is a discrete group realizing $\mathcal{F}$, and $\mathcal{C}$ is the collection of all $p$-centric subgroups in $G$. We prove that for every $\mathbb{Z}_{(p)} \mathcal{O} ^c (\mathcal{F})$-module $M$, the higher limits $\lim ^* _{\mathcal{O}^c (\mathcal{F})} M$ are isomorphic to the equivariant cohomology $H_ G ^* (|\mathcal{C}| ; \mathcal{M} )$, where $\mathcal{M}$ is the $G$-equivariant coefficient system on $|\mathcal{C}|$ induced by $M$. This is the fusion system version of a theorem by Grodal for higher limits over the $p$-orbit category $\mathcal{O}_p(G)$ of a finite group $G$.
As a consequence, for every fusion system $\cF$, we obtain a finite sequence of cohomology groups, involving the cohomology of $\mathcal{F}$, which is exact if and only if the mod-$p$ subgroup decomposition for $\mathcal{F}$ over $\mathcal{F}$-centric subgroups is sharp.
\end{abstract}



\maketitle

%------------------------

\setcounter{tocdepth}{1}
\tableofcontents
%\tableofcontents[hideallsubsections]

\section{Introduction and statement of results}\label{sect:Introduction}

Let $G$ be a discrete group and $\cC$ be a collection of subgroups of $G$ (closed under conjugation). The \emph{orbit category} $\cO _{\cC} (G)$ is the category whose objects are the subgroups $H \in \cC$, and whose morphisms from $H$ to $K$, with $H,K\in \cC$, are given by $G$-maps $G/H \to G/K$. For a commutative ring $R$ with unity, a functor $M: \cO _{\cC} (G) ^{op}\to R$-Mod
is called a (contravariant) \emph{$R\cO _{\cC} (G)$-module}. For example, for each $n\geq 0$, the functor $P \to H^n (P)$ that sends a subgroup $P\in \cC$ to its $n$-th mod-$p$ cohomology is an $\bbF_p \cO _{\cC} (G)$-module. 

For every $i\geq 0$, the \emph{higher limit} $\lim _{\cO _{\cC} (G)} ^i M$ of an $R\cO_{\cC} (G)$-module $M$ is defined as the $i$-th right derived functor of the limit functor $\lim _{\cO _{\cC} (G)} (-)$. The \emph{mod-$p$ subgroup decomposition} for the classifying space $BG$ over the collection $\cC$ is a mod-$p$ homology isomorphism
\[
\hocolim _{\cO_{\cC} (G)} \widetilde B \xrightarrow{} BG,
\]
where $\widetilde B: \cO_{\cC} (G)\to Top$ is a functor such that $\widetilde B(P)$ has the homotopy type of $BP$ for every $P\in \cC$. The mod-$p$ subgroup decomposition for $BG$ over $\cC$ is said to be \emph{sharp} if $$\underset{\cO _{\cC} (G)} {\lim {}^i} \ H^n (-) =0 $$ for every $i>0$ and $n\geq 0$. 
A $p$-subgroup $P$ of $G$ is \emph{$p$-centric} if $Z(P)$ is the Sylow $p$-subgroup of $C_G(P)$. The subgroup decomposition for a finite group $G$ is sharp for many collections $\cC$ such as the collection of all nontrivial $p$-subgroups of $G$, and the collection of all $p$-centric subgroups of $G$ (see \cite[Corollary 8.13]{Grodal-Higher} and \cite[Theorem 1.2]{Grodal-Smith}).

In \cite{Grodal-Higher}, Grodal proved a theorem that relates the higher limits over the orbit category $\cO _{\cC} (G)$ to the $G$-equivariant cohomology of the simplicial complex $|\cC|$ with coefficients in a $G$-equivariant coefficient system (see Theorem \ref{thm:IntroGrodal} below). We now introduce necessary definitions to state Grodal's theorem.

For a discrete group $G$, and a collection $\cC$ of subgroups of $G$, let $|\cC|$ denote the $G$-simplicial complex of chains of subgroups in $\cC$, where the $G$-action is given by conjugation. The equivariant simplex category $\Delta (|\cC|)_G$ of $|\cC|$ is defined by applying the Grothendieck construction to the simplex category of $|\cC|$  (see Section \ref{sect:Grodal's theorem} for details). For a commutative ring $R$ with unity, a functor $\cM : \Delta(|\cC|)_G \to R$-Mod is called an \emph{equivariant coefficient system} on $|\cC|$ over $R$. 

For a $G$-simplicial set $X$ and an equivariant coefficient system $\cM$, let $H^* _G (X; \cM)$ denote the equivariant cohomology of $X$ with coefficients in $\cM$ (see Section \ref{sect:EquivariantCoefSystem}). Grodal proved the following  theorem: 

\begin{theorem}[{Grodal \cite[Theorem 1.1]{Grodal-Higher}}]\label{thm:IntroGrodal} Let $\cC$ be a collection of $p$-subgroups of a finite group $G$, closed under passage to $p$-radical overgroups. Let $M: \cO_{\cC} (G) ^{op}\to \bbZp$-Mod be any functor, and let $\cM$ be the $G$-equivariant coefficient system on $|\cC|$ defined by precomposing $M$ with $\xi: \Delta (|\cC|)_G \to  \cO_{\cC}(G) ^{op}$ (see Definition \ref{def:xi}). Then there is an isomorphism $$\underset{\cO_{\cC}(G) }{\lim {}^*} M \cong H_ G ^* (|\cC| ; \cM ).$$
\end{theorem}

%The nondegenerate simplices of $|\cC|$ are the strict chains of subgroups in $\cC$.
%Passing to normalized cochains, we may take the product in the definition of $C_G^n (|\cC|; \cM)$ over the set $|\cC|_n$ of all strict chains $(P_0 <\cdots < P_n)$ with length $n$. 
Let $\cC$ be a collection satisfying the conditions of Theorem \ref{thm:IntroGrodal}, and assume that $BG$ has a mod-$p$ subgroup decomposition over $\cC$. Then as a consequence of Theorem \ref{thm:IntroGrodal} applied to $M=H^n(-)$, it is obtained that the subgroup decomposition is sharp if and only if the following sequences are exact for every $n\geq 0$: 
\begin{align}\label{eqn:GroupExactSeq-Subgroup}
0 \to H^n (G)  \to \bigoplus _{[P_0]\in |\cC|_0/G} H^n (P_0 ) ^{N_G(P_0)} \to \bigoplus _{[P_0< P_1] \in |\cC| _1/G} H^n (P_0) ^{N_G(P_0< P_1)} \to \cdots \to 0.
\end{align}

For any finite group $G$, there is a similar sequence involving the cohomology groups of the normalizers $N_G(\sigma)$ of chains of subgroups in $\cC$,
\begin{align}\label{eqn:GroupExactSeq-Normalizer}
0 \to H^n (G)  \to \bigoplus _{[P_0]\in |\cC|_0 /G} H^n (N_G(P_0)) \to \bigoplus _{[P_0< P_1] \in |\cC| _1/G} H^n ( N_G (P_0< P_1) )   \to \cdots \to 0. 
\end{align}
These sequences are exact for every $n\geq 0$ if and only if the mod-$p$ normalizer decomposition for $BG$ over $\cC$ is sharp (see \cite[Corollary 8.15]{Grodal-Higher}).

The main aim of this paper is to prove a version of Grodal's theorem for saturated fusion systems and obtain fusion system versions of the sequences in \eqref{eqn:GroupExactSeq-Subgroup} and \eqref{eqn:GroupExactSeq-Normalizer}. The sequences we obtain are exact if and only if the subgroup decomposition for the corresponding $p$-local finite group is sharp. 

To state the fusion system version of Grodal's theorem, we need to introduce more definitions. A saturated fusion system $\cF$ over a finite $p$-group $S$ is a category whose objects are subgroups of $S$, and whose morphisms satisfy certain axioms that mimic the Sylow theorems (see Definition \ref{def:Saturated}). 

\begin{definition}\label{def:OrbitFusion} For a saturated fusion system $\cF$ and a collection $\cC$ of subgroups of $S$ closed under $\cF$-conjugation, the \emph{orbit category} of $\cF$ over $\cC$ is the category $\cO _{\cC} (\cF)$ whose objects are the subgroups in $\cC$, and whose morphisms from $P$ to $Q$, where $P, Q\in \cC$, are given by $$\Mor _{\cO _{\cC} (\cF) } (P, Q) :=\mathrm{Inn} (Q) \backslash \Mor _{\cF} (P, Q)$$
where $\mathrm{Inn} (Q)$ denotes the group of inner automorphisms of $Q$ induced by conjugations by the elements in $Q$. 
\end{definition}
A subgroup $P\le S$ is \emph{$\cF$-centric} if $C_S(Q)\le Q$ for every subgroup $Q\le S$ isomorphic to $P$ in $\cF$.  The full subcategory of $\cO (\cF)$ whose objects are the $\cF$-centric subgroups of $S$ is called the \emph{centric orbit category} of $\cF$, denoted by $\cO^c (\cF)$.

Our approach for studying higher limits over the centric orbit category $\cO^c (\cF)$ is to replace this category with the fusion orbit category $\overline \cF_{\cC} (G)$ of the discrete group $G$ realizing the fusion system $\cF$. A finite $p$-subgroup $S$ of a discrete group $G$ is called a Sylow $p$-subgroup of $G$ if for every finite $p$-subgroup $P\le G$, there is an element $g\in G$ such that $gPg^{-1}\le S$. If $G$ has a Sylow $p$-subgroup $S$, then the associated fusion category $\cF_S (G)$ is defined as the category whose objects are subgroups of $S$, and whose morphisms from $P$ to $Q$, with $P, Q\le S$, are given by conjugations in $G$. More generally for a discrete group $G$ and a collection $\cC$, the fusion category of $G$ over $\cC$ is defined as the category $\cF_{\cC} (G)$ whose objects are the subgroups $H\in \cC$, and whose morphisms are given by conjugations in $G$.

\begin{definition}\label{def:FusionOrbit} Let $G$ be a discrete group and 
$\cC$ be a collection of subgroups of $G$.
The \emph{fusion orbit category} $\overline \cF _{\cC} (G)$ of $G$ is the quotient category of the orbit category $\cO_{\cC} (G)$ where the morphisms from $P$ to $Q$, with $P, Q\in \cC$, are given by
$$\Mor _{\overline \cF _{\cC} (G)} (P, Q) :=\Mor _{\cO _{\cC} (G)} (P, Q)/ C_G(P) =\mathrm{Inn}(Q)\backslash \Mor _{\cF_{\cC}} (P, Q).$$
\end{definition}

Leary-Stancu \cite{Leary-Stancu} and Robinson \cite{Robinson} showed that for every saturated fusion system $\cF$ over $S$, there is a discrete group $G$ with a Sylow $p$-subgroup $S$ such that $\cF_S(G)$ is isomorphic to $\cF$. 
In this case we say that the discrete group $G$ \emph{realizes the fusion system $\cF$}. If $G$ is a discrete group realizing a saturated fusion system $\cF$, and $\cC$ is the collection of $p$-centric subgroups in $G$, there is an equivalence of categories $\cO ^c (\cF) \simeq \overline \cF _{\cC} (G)$ (see Lemma \ref{lem:Equivalent}). Using this equivalence, we define a functor 
\[
\overline \xi: \Delta (|\cC|)_G \to  \cO ^c (\cF)^{op}
\]
as the composition of $\xi: \Delta (|\cC|)_G \to  \cO_{\cC} (G) ^{op}$ with the quotient functor $\cO _{\cC} (G)^{op} \to \overline \cF_{\cC} (G) ^{op} \simeq \cO ^c (\cF) ^{op}$. On objects, the functor $\overline \xi$ sends a chain $(P_0 \le \cdots \le P_n)$ in $|\cC|_n$ to $P'_0$ where $P_0'$ is the chosen subgroup of $S$ conjugate to $P_0$.

The main theorem of the paper is the following:

\begin{theorem}\label{thm:Main}
Let $G$ be a discrete group realizing a saturated fusion system $\cF$, and $\cC$ denote the collection of all $p$-centric subgroups in $G$. Suppose that $M$ is a $\bbZ_{(p)}\cO ^c \cF$-module, and $\cM$ is the equivariant coefficient system on $|\cC|$ defined by precomposing $M$ with $\overline \xi: \Delta (|\cC|)_G \to \cO ^c (\cF)^{op}$. 
Then, there is an isomorphism  $$\underset{\cO ^c (\cF) }{\lim {}^*} M \cong H_ G ^* (|\cC| ; \cM ).$$ 
\end{theorem}

Theorem \ref{thm:Main} is proved by using a proposition by Broto, Levi, and Oliver \cite[Proposition 3.2]{BLO2}. For a finite group $\Gamma$ and a $\bbZ_{(p)} \Gamma$-module $M$, let $\Lambda ^* (\Gamma, M)$ denote the group of higher limits $\lim ^* _{\cO _p (\Gamma)} F_M$ where $F_M$ is the atomic functor concentrated at the trivial group with $F_M(1)=M$. Broto, Levi, and Oliver \cite[Proposition 3.2]{BLO2} prove that if $M: \cO ^c (\cF)^{op}\to \bbZ_{(p)}$-Mod is an atomic functor concentrated at $Q\in \cF^c$, then the higher limit $\lim ^* _{\cO ^c (\cF)} M$ is isomorphic to $\Lambda ^* (\Out _{\cF} (Q), M(Q) )$, where $\Out _{\cF} (Q):=\Aut _{\cF } (Q)/\Inn (Q)$. 

To prove the isomorphism in Theorem \ref{thm:Main}, 
we construct a natural map $\Psi^*: H_ G ^* (|\cC| ; \cM )
\to \lim^* _{\cO^c (\cF)} M$,
and prove that it is an isomorphism when $M$ is atomic by identifying both sides with
$\Lambda^*(\Out_{\cF}(Q),M(Q))$.  
The general case is proved by induction, using a filtration of $M$ with sections isomorphic to atomic functors.

By a theorem of Chermak \cite{Chermak-2013}, every saturated fusion system $\cF$ over $S$ has an associated centric linking system $\cL$, unique up to isomorphism (see \cite[Definition 1.7]{BLO2}). The triple $(S,\cF,\cL)$ is
called a $p$-local finite group associated to $\cF$.
The classifying space of a $p$-local finite group $(S, \cF, \cL)$ is defined as the $p$-completion $|\cL|_p ^{\wedge}$ of the geometric realization of $\cL$. There is a subgroup decomposition 
\[
|\cL| \simeq \underset{\cO ^c (\cF) } \hocolim \widetilde B
\]
for the centric linking system $\cL$ where $\widetilde B: \cO^c (\cF) \to Top$ is a homotopy lifting of the homotopy functor $P\to BP$ \cite[Proposition 2.2]{BLO2}.  The Bousfield-Kan spectral sequence associated to this homology decomposition is of the form
\[
E_2 ^{p,q}=\underset{\cO ^c (\cF)}{\lim {}^p} H^q (-) \Rightarrow H^{p+q} (|\cL | ; \bbF_p).
\]
\begin{definition}\label{def:BFSharp} The subgroup decomposition for $\cF$ over $\cF$-centric subgroups is \emph{sharp} if 
\[
\underset{\cO^c (\cF) } 
{\lim {}^i} H^n (-)=0
\]
for all $i>0$ and $n \geq 0$. 
\end{definition}

It is conjectured by Diaz and Park \cite[Conjecture]{DiazPark} that for every saturated fusion system $\cF$, the subgroup decomposition for $\cF$ over $\cF$-centric subgroups is sharp. 
In fact, they conjectured that the higher limits over the centric orbit category vanish for the contravariant part of a Mackey functor for $\cF$ with coefficients in $\bbF_p$, defined on the full orbit category of $\cF$. They proved their conjecture for fusion systems realized by finite groups (see \cite[Theorem B]{DiazPark}). The sharpness of the subgroup decomposition has been proved in many cases (see, for example, \cite{DiazPark}, \cite{GrazianMarmo}, \cite{GlaubermanLynd}, \cite{Bova-2026}), but the general case is still open.

As a consequence of Theorem \ref{thm:Main}, we have the following corollary:

\begin{corollary}\label{cor:Main}
Let $\cF$ be a saturated fusion system over a finite $p$-group $S$, and $\cC$ be the collection of all $\cF$-centric subgroups of $S$. Then, the subgroup decomposition for $\cF$ over $\cF$-centric subgroups is sharp if and only if the following sequence is exact for all $n\geq 0$:
\begin{align*} 
\quad 0 \to H^n (\cF )  \to \bigoplus _{[P_0] \in |\cC|_0/\cF} H^n (P_0) ^{\Aut _\cF(P_0)} \to \bigoplus _{[P_0< P_1] \in |\cC| _1/\cF} H^n (P_0) ^{\Aut_\cF(P_0< P_1)} \to \cdots \to 0.
\end{align*}
\end{corollary}

This sequence can be considered as the fusion system version of the sequence given in \eqref{eqn:GroupExactSeq-Subgroup}.

In Section \ref{sect:NormalizerDec}, we discuss the normalizer decompositions for $p$-local finite groups, introduced by Libman \cite{Libman-NormDec}, and obtain a sequence of cohomology groups analogous to the normalizer sequence given in \eqref{eqn:GroupExactSeq-Normalizer}. Let $(S, \cF, \cL)$ be a $p$-local finite group. The \emph{normalizer decomposition} for $\cF$ is a homotopy equivalence 
\[
\hocolim _{\overline sd (\cF^c)} \delta \xrightarrow{} |\cL |
\]
where $\delta: \overline sd (\cF^c ) \to Spaces$ is a functor that sends an $\cF$-conjugacy class $[\sigma]$ of a chain $\sigma =(P_0<\dots < P_n)$ to a space $\delta ([\sigma])$ homotopy equivalent to $B\Aut _{\cL} (\sigma)$ (see Section \ref{sect:NormalizerDec} for details).
The normalizer decomposition for $\cF$ is said to be \emph{sharp} if \[
\underset{\overline sd(\cF^c)}{\lim {}^i} H^n ( B\Aut _{\cL} (\sigma))=0
\]
for all $i>0$ and $n\geq 0$. 

For a chain $\sigma =(P_0<\cdots< P_n)$ of subgroups of $S$, let $N_{\cF} (\sigma)$ denote the normalizer fusion system for $\sigma$ (see Definition \ref{def:NormFusion}).
In Section \ref{sect:NormalizerDec}, we show that for every fully $\cF$-normalized chain $\sigma$ of $\cF$-centric subgroups of $S$, there is an isomorphism $H^*(B\Aut_{\cL} (\sigma))\cong H^*(N_{\cF} (\sigma))$.
In \cite[Theorem 1.9]{Yalcin-Higher}, we proved that  over the collection of $\cF$-centric subgroups, the subgroup decomposition for $\cF$ is sharp if and only if the normalizer decomposition for $\cF$ is sharp. As a consequence of these, we prove the following: 

\begin{corollary}\label{cor:MainNorm}
Let $\cF$ be a saturated fusion system over $S$, and $\cC$ be the collection of all $\cF$-centric subgroups in $S$. Then the subgroup decomposition for $\cF$ over $\cC$ is sharp if and only if the following sequence is exact for every $n \geq 0$:
\begin{align*} 
\quad 0 \to H^n (\cF )  \to \bigoplus _{[P_0] \in |\cC|_0/\cF} H^n (N_\cF(P_0))\to \bigoplus _{[P_0< P_1] \in |\cC| _1/\cF} H^n (N_\cF(P_0< P_1)) \to \cdots \to 0
\end{align*}
where the sums are taken over fully $\cF$-normalized representatives.
\end{corollary}

The sequences given in Corollaries \ref{cor:Main} and \ref{cor:MainNorm} provide alternative approaches for proving the sharpness over the collection of $\cF$-centric subgroups.
In the cases where the subgroup decomposition over the $\cF$-centric subgroups is sharp, they provide exact sequences for calculating the cohomology of the fusion system $\cF$.

\vskip 10 pt
\noindent
{\bf Notation and Conventions:}
When we say $\cC$ is a collection of subgroups of $G$, it always means that $\cC$ is closed under $G$-conjugation. $R$ always denotes a commutative ring with unity. $H^*(-)$ denotes cohomology with $\mathbb F_p$-coefficients unless otherwise specified. This paper is written from the point of view that a space is a simplicial set. Throughout the paper, all groups are discrete groups, all $p$-groups are finite, and all fusion systems are saturated. 

\vskip 10 pt
\noindent
{\bf Acknowledgments:}
I thank Caroline Yal\c c\i n for reading the paper, commenting on the style, and pointing out typos and grammatical errors. 

\vskip 10 pt
\noindent
{\bf Declaration of AI usage:}
The results in this paper are obtained without AI usage, and the paper is written by the author. After preparing the first draft, ChatGPT (GPT-6 Astra) was used to proofread the paper and for checking the proofs. Based on the feedback from ChatGPT, I revised Section \ref{sect:ProofMain} and added a detailed proof of the commutativity of the diagram in \eqref{eqn:CommDiagram}. I take full responsibility for the final version.

%-----------------

\section{Modules over the orbit category and Bredon cohomology}\label{sect:ModulesAndUniversal}

This is a preliminary section where we introduce definitions for the higher limits over the orbit category, universal space $E_{\cC}G$, and Bredon cohomology. We also state some results that will be used later in the paper.

\subsection{Higher limits over the orbit category}\label{Sect:HigherLimits}
Let $\bC$ be a small category and $R$ be a commutative ring with unity. A functor $M: \bC ^{op}\to R$-Mod is called a (contravariant) $R\bC$-module. 
The category of $R\bC$-modules is an abelian category with enough projective and injective modules. For $R\bC$-modules 
$M$ and $N$, the ext-groups $\Ext^i _{R\bC} (M, N)$ for $i\geq 0$ are defined in the usual way, using projective resolutions for $M$ or injective resolutions for $N$.

Let $\underline R$ denote the constant $R\bC$-module. For every $i\geq 0$, the cohomology $H^i (\bC; M)$ of the small category $\bC$ with coefficients in the $R\bC$-module $M$ is defined to be the ext-group $\Ext ^i _{R\bC} (\underline R, M)$. For every $i\geq 0$, the higher limit $\lim ^i _{\bC} M$ of an $R\bC$-module $M$ is defined as the $i$-th right derived functor of the limit functor $\lim _{\cC} (-)$. Note that $H^* (\bC; M)\cong \lim ^* _{\bC} M$. 

Let $G$ be a discrete group and $\cC$ be a collection of subgroups of $G$.
Modules over the orbit category $\cO _{\cC} (G)$, and the higher limits $\lim^i _{\cO_{\cC} (G)} M$ are defined by taking $\bC=\cO_{\cC} (G)$ in the above definitions. For each $H \in \cC$, let $R[G/H^?]$ denote the 
$R\cO _{\cC} (G)$-module that sends each $K\in \cC$ to the $R$-module
$$R[(G/H)^K ]\cong R\Map_G (G/K, G/H) \cong R\Mor _{\cO _{\cC} (G)} (K, H).$$
An $R\cO_{\cC} (G)$-module which is a direct sum of modules of the form $R[G/H^?]$ is called a free $R\cO _{\cC} (G)$-module. By the Yoneda lemma, free $R\cO _{\cC} G$-modules are projective, and they provide enough projectives for constructing projective resolutions for any $R\cO_{\cC}(G)$-module.

%------------------

\subsection{Universal space over a collection}\label{sect:Universal}

Let $G$ be a discrete group and $\cC$ be a collection of subgroups of $G$.
The collection $\cC$ is closed under conjugation, but we do not assume that it is closed under passing to subgroups.

\begin{definition} Consider the category $\cE _{\cC} G$ whose objects are the pairs $(G/H, x)$ where $H\in \cC$ and $x\in G/H$, and whose morphisms $(G/H, x) \to (G/K, y) $ are given by the $G$-map $f: G/H \to G/K$ such that $f(x)=y$ .
The \emph{universal space} $E_{\cC} G$ is defined as the nerve of the category $\cE _{\cC} G$.
\end{definition}

The category $\cE_{\cC} G$ is obtained by applying the Grothendieck construction to the functor $F : \cO_{\cC} (G) \to G$-Sets that sends $H\in \cC$ to the transitive $G$-set $G/H$. By Thomason's homotopy colimit theorem \cite[Theorem 1.2]{Thomason-1979}, the classifying space $E_{\cC} G$ is homotopy equivalent to the homotopy colimit $\hocolim _{\cO_{\cC} G} F$. The following result is useful for characterizing $E_{\cC} G$ as a $G$-space.
 
\begin{lemma}[{\cite[Proposition 2.14]{Dwyer-HomDec}}]\label{lem:EG^K}
Consider the functor $f: \cE _{\cC} G \to \cC$ that sends each object $(G/H, gH)$ in $\cE_{\cC} G$ to ${}^gH:=gHg^{-1}$ in $\cC$. The realization of $f$ gives a $G$-map $E_{\cC} G \to |\cC|$ that has the property that for every subgroup $K \le G$, the $N_G(K)/K$-equivariant map $$(E_{\cC} G)^K \to |\cC_{\geq K}| $$ is a homotopy equivalence of spaces.  
\end{lemma}

% \begin{proof}
% Consider the functor $f: \cE _{\cC} G \to \cC$ that sends $(G/H, gH)$ to $gHg^{-1}$ for every  object $(G/H, gH)$ in $\cE_{\cC} G$. For every $K\le G$, $(G/H , gH)$ is fixed by $K$ if $g^{-1} K g \le H$. Hence the image of $(\cE _{\cC} G)^K$ lies in $\cC _{\geq K}$. The functor $g : \cC\to \cE _{\cC} G$ defined by $H \to (G/H, H)$ sends $\cC_{\geq K}$ to $(\cE _{\cC} G) ^K $. It is easy to see that $f\circ g=\id$ and that $g \circ f$ is naturally equivalent to the identity functor. Hence  $f$ induces a homotopy equivalence $(E_{\cC} G)^K \to |\cC_{\geq K}|$. It is clear from the definitions that for every $K\le G$, the map $f$ restricted to $(E_{\cC} G)^K$ is an $N_G(K)/K$-equivariant map. %When $K \in \cC$, the realization of the poset $\cC_{\geq K}$ is contractible since it has a minimum element. 
% \end{proof}

Note that since $E _{\cC} G$ is a $G$-simplicial set, the topological realization of $E_{\cC} G$ has a structure of a $G$-CW-complex. 
We denote the topological realization of $E_{\cC} G$ also by $E_{\cC} G$. As a $G$-CW-complex, the universal space $E_{\cC} G$ has the following properties.

\begin{lemma}\label{lem:Properties} Let $G$ be a discrete group and $\cC$ be a collection of subgroups of $G$. Then,
\begin{enumerate}
\item all isotropy subgroups of cells in $E_{\cC}G $ are in the collection $\cC$, and
\item for every $K \in \cC$, the fixed point set $(E_{\cC} G) ^K$ is contractible.
\end{enumerate}
\end{lemma}

\begin{proof} $E_{\cC} G$ is a $G$-CW-complex whose cells are the nondegenerate simplices in the nerve of $\cE_{\cC} G$. Such simplices are given by a sequence of maps $$(G/H_0, g_0H_0 )\maprt{f_1} (G/H_1, g_1 H_1) \maprt{f_2} \cdots \maprt{f_n} (G/H_n , g_n H_n )$$ with non-identity morphisms $f_i$ between the objects of $\cE_{\cC} G$. 
Since for each $i \in \{0, \dots, n-1\}$, we have $f_{i+1} (g_i H_i) =g_{i+1} H_{i+1}$, the stabilizer of such a chain is equal to the stabilizer of the first term which is 
$g_0 H_0 g_0 ^{-1}$. Note that since $\cC$ is closed under conjugation, this subgroup lies in $\cC$. Hence the stabilizers of the cells of $E_{\cC}G$ lie in $\cC$. 

For every $K \in \cC$, the simplicial complex $|\cC_{\geq K}|$ is contractible since it has a smallest element. Hence by Lemma \ref{lem:EG^K}, we can conclude that for every $K \in \cC$, the fixed point set $(E_{\cC} G )^K$ is contractible.
\end{proof}

The universal space $E_{\cC} G$  is sometimes defined as the $G$-CW-complex that satisfies the above conditions.  Note that using these properties we can prove that for every $G$-CW-complex $X$ with isotropy subgroups in $\cC$, there is a unique map $X \to E_{\cC} G$ up to $G$-homotopy.
%This is why the space $E_{\cC}G$ is called the universal space. 

\subsection{Bredon cohomology of a $G$-CW-complex}\label{sect:SubBredon}

In this section we introduce the Bredon cohomology of a $G$-complex introduced by Bredon \cite{Bredon-Book}. Let $G$ be a discrete group and $\cC$ be a collection of subgroups of $G$. For a $G$-set $X$, let $R[X^?]$ denote the $R\cO_{\cC} (G)$-module defined by $$P \to R[X^P] \cong R \Map _G (G/P, X)$$ for every $P\in \cC$. If all the isotropy subgroups of $X$ are in $\cC$, then $R[X^?]$ is a free $R\cO_{\cC} (G)$-module, hence it is projective.

\begin{definition} For a $G$-CW-complex $X$, let 
$C_* (X ^?; R)$ denote the chain complex of $\cO _{\cC} (G)$-modules whose value 
at $P\in \cC$ is equal to the cellular chain complex $C_* (X^P; R)$ of the fixed point subspace $X^P$. Note that if $X_n$ is the $G$-set of $n$-cells in $X$, then $C_n (X^?; R) \cong R[X_n ^?]$ as $R\cO_{\cC} (G)$-modules. 
\end{definition}

%If $X$ is a $G$-CW-complex and $Y$ is a subcomplex of $X$ invariant under $G$-action, then we say $(X, Y)$ is a pair of $G$-CW-complexes. 
For a pair of $G$-CW-complexes $(X, Y)$, we define the associated chain complex of $R\cO_{\cC} (G)$-modules to be the quotient complex
$$C_* (X^? , Y^?; R) := C_* (X^?; R)/C_* (Y^?; R).$$ 
The augmented complex $\widetilde C_* (X ^?; R)$ is defined to be the complex $(C_* (X^?; R) \maprt{\epsilon} \underline R)$ where  $\epsilon$ is the augmentation map and the constant functor $\underline R$ is taken to be at dimension $-1$. 
The following lemma is well-known and its proof easily follows from Lemma \ref{lem:Properties}.

\begin{lemma}\label{lem:ProjRes} Let $E_{\cC}G$ be the universal space for $G$ over $\cC$. Then the augmented chain complex $\widetilde C_* ( E_{\cC} G ^?; R)$ is a projective resolution for the constant functor $\underline R$ as an $R\cO _{\cC} (G)$-module. 
\end{lemma}

% \begin{proof} This follows from Lemma \ref{lem:Properties} and from the fact that for each $G$-set $X$, the module $R[X^?]$ is a projective $R\cO_{\cC} (G)$-module when all the isotropy subgroups of $X$ are in $\cC$.  
% \end{proof}

The orbit category of a group $G$ over the collection of all subgroups of $G$ is denoted by $\cO (G)$. Note that over $\cO(G)$, the chain complex $C_*(X^?, Y^?; R)$ is a chain complex of projective modules.

\begin{definition}\label{def:BredonCoh} 
Let $(X,Y)$ be a pair of $G$-CW-complexes and $M$ be an $R\cO (G)$-module.
The \emph{Bredon cohomology}  of $(X, Y)$ with coefficients in $M$ is defined by 
$$\underline H^* _G  (X^?, Y^?; M) := H^* (\Hom _{R\cO (G) } ( C_* (X^?, Y^?; R ), M) ).$$
The Bredon cohomology $\underline{H} ^* _G (X^? ; M)$ of a $G$-CW-complex $X$ is the Bredon cohomology of the pair $(X, \emptyset)$. 
\end{definition}

Note that if $M$ is an $R\cO_{\cC} (G)$-module, by taking its right Kan extension via the inclusion functor $j^{op}: \cO_{\cC} (G)^{op}\to \cO (G)^{op}$, the module $M$ can be extended to an $R\cO (G)$-module $\widetilde M$.  Also note that the complex $C_* (X^?, Y^?;R) $ of $R\cO (G)$-modules can be considered as a chain complex of $R\cO_{\cC} (G)$-modules by restricting all the chain modules via the inclusion functor $\cO_{\cC} (G) \to \cO (G)$. %We often denote this restricted module also by $C_* (X^?, Y^?)$. 
Under these conventions, for any $R\cO _{\cC} (G)$-module $M$, there is an isomorphism
\begin{equation}\label{eqn:Isomorphism}
\Hom _{R\cO (G)} (C_* (X^?, Y^?), \widetilde M) \cong 
\Hom _{R\cO_{\cC} (G)} (C_* (X^?, Y^? ) , M)
\end{equation}
which follows by the adjointness property of right Kan extensions.
%The higher limits over the orbit category $\cO_{\cC} (G)$ can be expressed as the Bredon cohomology of the universal space.

\begin{proposition}[Grodal \cite{Grodal-Higher}]\label{pro:LimBredon} Let $G$ be a discrete group and $\cC$ be a collection of subgroups of $G$. For every $R\cO_{\cC} (G)$-module $M$, there is an isomorphism
$$\underset{\cO_{\cC} (G)} {\lim {}^*}  M \cong {\underline H} _G ^* (E_{\cC} G ^?; \widetilde M).$$
\end{proposition}

\begin{proof}
By Lemma  \ref{lem:ProjRes} the augmented chain complex $\widetilde C_* (E_{\cC} G ^?, R)$ defines a projective resolution of $\underline R$ as an $R\cO_{\cC} (G)$-module. Then for every $R\cO _{\cC}(G)$-module $M$, we have
$$\underset{\cO _{\cC} (G)}{\lim} {}^*  M \cong H^* (\Hom _{R\cO_{\cC}(G) } (C_* (E_{\cC} G ^? ; R ), M )).$$ Using the isomorphism given in \eqref{eqn:Isomorphism}, we conclude that this cohomology module is isomorphic to the Bredon cohomology $\underline H _G ^* (E_{\cC} G ^?; \widetilde M)$. 
\end{proof}    

\begin{remark} Since the restriction of $\widetilde M$ to $R\cO_{\cC} G$ is isomorphic to $M$, and all the isotropy subgroups of $E_{\cC}G$ are in $\cC$, the Bredon cohomology $\underline H^*_G (E_{\cC} G^?; \widetilde M)$ is sometimes denoted by $\underline H^*_G (E_{\cC} G^?; M)$.
\end{remark}

An $R\cO_{\cC} (G)$-module $M$ is called an \emph{atomic functor} concentrated at $K \in \cC$ if for every subgroup $H \in \cC$ that is not $G$-conjugate to $K$, we have $M(H)=0$.  Let $\cO_p (G)$ denote the orbit category over the collection of all $p$-subgroups of $G$.
For an $RG$-module $M$, let $F_M$ denote the atomic $R\cO_p (G)$-module  concentrated at the trivial subgroup, with $F_M(1)=M$. The higher limits $\lim^* _{\cO_p (G)} F_M$ of the atomic functor $F_M$ are denoted by $\Lambda ^* (G; M)$. The following is proved by Jackowski, McClure, and Oliver \cite{JMO}.

\begin{proposition}[{\cite[Lemma 5.4]{JMO}}]
Let $G$ be a finite group and  $M: \cO_p (G)^{op} \to \bbZ_{(p)}$-Mod be an atomic functor concentrated at the $p$-subgroup $P\le G$. Then there is an isomorphism $$\underset{\cO_p (G)} {\lim {}^*} M \cong \Lambda ^* (WP , M(P) )$$    where $WP = N_G(P)/P$.
\end{proposition}

%In the next section we introduce the equivariant cohomology of a $G$-space.

%-----------------

\section{Equivariant cohomology and equivariant coefficient systems}\label{sect:Grodal's theorem}

In this section, we give definitions related to equivariant cohomology and equivariant coefficient systems. We also introduce the equivariant cohomology of a $G$-simplicial set with coefficients in an equivariant coefficient system.

\subsection{Equivariant cohomology of a $G$-space}
Let $G$ be a discrete group and $R$ be a commutative ring with unity. For a pair of $G$-CW-complexes $(X, Y)$, let $C_* (X, Y; R)$ denote the cellular chain complex of the pair $(X, Y)$ with coefficients in $R$. Note that $C_* (X, Y;R)$ is a chain complex of $RG$-modules.

\begin{definition}\label{def:EqCoh} For a pair of $G$-CW-complexes
$(X,Y)$ and an $RG$-module $M$, the \emph{$G$-equivariant cohomology} $H^*_G (X, Y ; M)$ is defined as the cohomology of the cochain complex $\Hom _{RG} (C_* (X, Y; R) , M )$. 
\end{definition}

For a $G$-CW-complex $X$, let $X^s=\bigcup _{1< H \le G} X^H$ denote the singular set of $X$. For a group $G$, let $S_p^e (G)$ denote the poset of all $p$-subgroups of $G$, and $S_p (G)$ denote the poset of all nontrivial $p$-subgroups of $G$. We have the following observation due to Grodal \cite{Grodal-Higher} which will be used later in the paper.

\begin{proposition}\label{pro:CohIsom}
Let $G$ be a finite group, and $E_pG$ denote the universal space of $G$ over the collection of all $p$-subgroups in $G$. For every $\bbZ _{(p)}G$-module $M$, there is an isomorphism
$$ H^* _G (E_p G , (E_p G)^s ; M) \cong H^* _G ( |S_p^e  (G)|, |S_p (G)| ; M).$$
\end{proposition}

\begin{proof}
Consider the $G$-map of pairs  
$$f: (E_p G , (E_p G) ^s ) \to (|S_p ^e  (G)|, |S_p (G)|)$$ 
induced by the functor $\cE _p G \to S_p ^e (G)$ that sends $(G/P, gP)$ to ${}^g P$. By Lemma \ref{lem:EG^K} applied with $K=1$, the universal space $E_p G$ is contractible. The realization of the poset $S_p^e (G)$ is also contractible since it has a minimal element. Hence the map $f:E_pG \to S_p^e(G)$ is a homotopy equivalence. 
The singular set $(E_p G)^s$ is $E_{\cC} G$ where $\cC$ is the collection of all nontrivial $p$-subgroups of $G$. The restriction of $f$ to $E_{\cC} G$ is the map defined in Lemma \ref{lem:EG^K}. Applying Lemma \ref{lem:EG^K} with $K=1$, we obtain that the restriction of $f$ to $(E_pG)^s$ is a homotopy equivalence. Hence $f$ is a (nonequivariant) homotopy equivalence of pairs. 

Let $S$ be a Sylow $p$-subgroup of $G$. For every subgroup $1 \neq P\le S$,  the pair $$((E_p G)^P, (E_p G ^s)^P)$$ is a trivial pair with both spaces contractible by Lemma \ref{lem:Properties}, and the fixed point sets $|S^e_p (G)|^P$ and $|S_p (G)| ^P $ are both contractible by \cite[Lemma 3.3.20]{Smith-2011}. For $P=1$ we have a homotopy equivalence of pairs by the above argument. It follows that $f$ is an $S$-homotopy equivalence, hence it induces an isomorphism
$$ H^* _S (E_p G , (E_p G)^s ; M) \cong H^* _S ( |S_p^e  (G)|, |S_p (G)| ; M).$$
Since $M$ is a $\bbZ _{(p)}G$-module, by a standard transfer argument, we can conclude that $f$ induces an isomorphism on $G$-equivariant cohomology groups.
\end{proof}

%----------------

\subsection{Equivariant coefficient systems}\label{sect:EquivariantCoefSystem}

%The \emph{simplex category} $\Delta$ is the category whose objects are sets 
%$[n]=\{0,\dots, n\}$, for $n\geq 0$, and whose morphisms are order preserving functions $\lambda : [m]\to [n]$. 
A simplicial $G$-set $X$ is a functor $X: \Delta^{op} \to G$-sets, where $\Delta$ denotes the simplex category. For a $G$-simplicial set $X$, the \emph{category of simplices} in $X$ is the category $\Delta(X)$ whose objects are the simplices of $X$, where for every $\tau\in X_m$ and $\sigma \in X_n$, the morphisms $\tau \to \sigma$ in $\Delta(X)$ are given by the morphisms $\lambda : [m]\to [n]$ in $\Delta$ such that $\lambda ^* (\sigma)=\tau$. 

If $X$ is a $G$-simplicial set, then the category $\Delta(X)$ is a $G$-category. Let $\Delta (X)_G$ denote the Grothendieck construction $\int _G \Delta (X)$ which is the category whose objects are the simplices of $X$, and for $\tau\in X_m$ and $\sigma \in X_n$, 
the morphisms $\tau \to \sigma$ in $\Delta (X)_G$ are given by the pairs $(g, \lambda)$ such that $g\in G$ and $\lambda :[m]\to [n]$ such that
$\lambda ^* (\sigma)=g\tau$. A (covariant) functor 
\[
\cM: \Delta  (X) _G \to R\text{-Mod}
\]
is called a \emph{$G$-equivariant coefficient system} on $X$ over $R$. Note that for each $i \in [n]$, the coface map $d^i : [n-1]\to [n]$ induces a homomorphism $\cM (1, d^i): \cM (d_i \sigma ) \to \cM (\sigma)$ where $d_i : X_n \to X_{n-1}$ is the corresponding face map. 

Given a $G$-simplicial set $X$ and a $G$-equivariant coefficient system $\cM$ on $X$, let $C^* _G (X; \cM)$ denote the cochain complex such that for each $n\geq 0$, 
$$ C^n _G (X; \cM ) = \Bigl ( \prod _{\sigma \in X_n } \cM (\sigma) \Bigr ) ^G.$$ 
 %The elements of $C^n (X; \cM)=\prod _{\sigma \in X_n} \cM (\sigma)$ are the functions \[f : X_n \to \bigcup \limits_{\sigma \in X_n } \cM (\sigma)\]
%such that $f (\sigma ) \in \cM (\sigma)$ for every $\sigma\in X_n$. 
The coboundary map $\delta ^n : C^n _G (X; \cM) \to C^{n+1} _G (X; \cM )$ is defined by $$ (\delta ^n f ) (\sigma)= \sum _{i=0} ^{n+1} (-1)^i \cM ( 1, d^i) ( f( d_i \sigma))$$
for every $f\in C_G ^n (X; \cM)$ and $\sigma\in X_{n+1}$. The $G$-action on $C^n (X; \cM)=\prod _{\sigma \in X_n } \cM (\sigma )$ is defined as follows: For $g\in G$, $f\in C^n (X; \cM)$, and $\sigma \in X_n$, we define $$(gf) (\sigma )=\cM (g, \id_\sigma ) (f(g^{-1} \sigma ) )$$ where $\cM (g, \id_\sigma): \cM (g^{-1} \sigma ) \to \cM (\sigma )$ is the homomorphism induced by the morphism $(g, \id_\sigma ) : g^{-1} \sigma \to \sigma$ in $\Delta (X)_G$.

\begin{definition}\label{def:EquivCoh}
Given a $G$-simplicial set $X$ and a $G$-equivariant coefficient system $\cM$ on $X$, the \emph{equivariant cohomology $H^n _G(X; \cM) $ with coefficients in $\cM$}
is the $n$-th cohomology of the cochain complex $C^* _G ( X; \cM)$ introduced above.
\end{definition}

In our applications, the simplicial set $X$ is the simplicial set associated to the simplicial complex $|\cC|$, where $\cC$ is a collection of subgroups in a discrete group $G$. The equivariant coefficient system $\cM$ on $|\cC|$ is associated to an $R\cO_{\cC} (G)$-module $M$ via the functor $\xi$ defined below.
%obtained by precomposing an $R\cO_{\cC}(G)$-module $M$ with the functor $\xi: \Delta(|\cC|)_G \to \cO_{\cC} (G) ^{op}$. 
In this case, the normalized cochains of $C_G ^* (X; \cM)$ are indexed by strict chains in $\cC$ and the subcomplex of normalized cochains has the same cohomology as the cohomology of $C_G^* (X; \cM)$.

\subsection{The functor $\xi$ and Grodal's theorem}

Let $G$ be a discrete group and $\cC$ be a collection of subgroups of $G$. As before let $|\cC|$ denote the $G$-simplicial complex whose simplices are the strict chains of subgroups in $\cC$. We can consider $|\cC|$ as a simplicial set in the usual way.

\begin{definition}\label{def:xi} 
The functor $\xi: \Delta (|\cC| )_G \to \cO _{\cC} (G)^{op}$ is defined as the composition 
\[
\Delta (|\cC| )_G \to (\cC ^{op} )_G \to \cO _{\cC}(G)^{op}
\]
where the first functor sends a simplex $\sigma=(H_0 \le \cdots \le H_n)$ of $|\cC|$ to $H_0 \in \cC$, and the second functor sends $H\in \cC$ to itself in $\cO _{\cC} (G)$. Given a morphism $(g, \lambda): \tau \to \sigma$ in $\Delta (|\cC|)_G$
with $\tau = (K_0 \le \cdots \le K_m)$ and $\sigma = (H_0 \le \cdots \le H_n)$, we have $H_0 \le H_{\lambda(0)}={}^g K_0$. This gives a morphism $\mu:{}^g K_0 \to H_0$ in $\cC^{op}$.
The first functor sends $(g, \lambda)$ to $(g,\mu): K_0  \to H_0$ in $(\cC^{op})_G$. The second functor $(\cC^{op})_G \to \cO_{\cC} (G)^{op}$ sends the morphism $(g, \mu): K \to H$ in $(\cC^{op} )_G$ to the $G$-map $G/H \to G/K$ defined by $H \to gK$.  
\end{definition}

Note that if $(g, \mu): K\to H$ is a morphism in $(\cC^{op})_G$, then $H\le {}^g K$ which gives $H^g\le K$. Hence the assignment $H\to gK$ defines a $G$-map $G/H \to G/K$. Given an $R\cO _\cC (G)$-module $M$, the composition of $M$ with the functor $\xi$ gives a $G$-equivariant coefficient system $\cM$ on the simplicial complex $|\cC|$. If $G$ is a finite group, $\cC$ is a collection of $p$-subgroups closed under passage to $p$-radical overgroups, and $M$ is a $\bbZ_{(p)}\cO_{\cC}(G)$-module, Grodal’s theorem (Theorem \ref{thm:IntroGrodal}) states that there is an isomorphism 
\[
\underset{\cO_{\cC} (G)}{\lim {}^*}  M \cong H_G^*(|\cC|;\cM).
\]
 
The following example shows that this isomorphism can fail for infinite groups.

\begin{example}\label{ex:InfiniteGroup}
Consider the infinite group 
\[
G=\langle a, b \, |\, b^3=1, aba^{-1}=b^{-1}\rangle \cong C_3 \rtimes \bbZ.
\]
Note that $S=\langle b \rangle \cong C_3$ is the unique Sylow $3$-subgroup of $G$. Thus $\cC=\{ S\}$ and the centric orbit category 
$\cO ^c (G)$  has only one object $S$ with normalizer $N_G(S)\cong G$.
Let $R=\bbZ_{(3)}$ and $M$ be an $R\cO^c (G)$-module. If $\cM$ is the $G$-equivariant coefficient system on $|\cC|$ induced by $M$, then
\[
H^i _G (|\cC| ; \cM) = \begin{cases} M(S)^G   & \text{ if } i=0, \\ 0 & \text{ otherwise.}\end{cases}
\]
To calculate the higher limits $\lim ^* _{\cO^c(G)} M$, observe that  a resolution of the constant functor $\underline R$ as an $R\cO^c (G)$-module can be given as 
\[
0 \to R[(G/S)^?] \maprt{1-a} R[(G/S)^?] \to \underline R \to 0.
\]
Hence, we have 
\[
\underset{\cO ^c (G)}{\lim {}^i} M =\begin{cases} \ker \{ M(S ) \maprt{1-a} M(S) \} & \text{ if } i=0 \\
\mathrm{coker} \{ M(S) \maprt{1-a} M(S) \} & \text{ if } i=1 \\
0 & \text{otherwise}.
\end{cases}
\]
If the action of the element $a \in G$ on $M(S)$ is trivial,  then $\lim ^i _{\cO ^c (G) } M \cong M(S)$ for $i=0,1$. 
In particular, if $M(S)=R$ is the trivial module, then 
\[
R \cong \underset{\cO ^c (G)}{\lim {}^1} M \not \cong H^1_G ( |\cC|; \cM)=0.
\]
This shows that the isomorphism in Theorem \ref{thm:IntroGrodal} does not hold in general when $G$ is an infinite group. 

Note that in this example, the fusion system $\cF=\cF_S(G)$ for $G$ is isomorphic to the fusion system of the symmetric group $S_3$, and the statement of Theorem \ref{thm:Main} still holds for $\cF$. 
To see this, let $M$ be an $R\cO^c (\cF)$-module. The automorphism group $\Aut _{\cO ^c (\cF) } (S)=N_G(S)/SC_G(S)$ is isomorphic to $C_2$, hence $\underline R$ is a projective $R\cO^c (\cF)$-module. This gives that $\lim ^i _{\cO^c (\cF)} M =0$ for $i>0$, and 
\[
\underset{\cO^c(\cF)}{\lim {}^0} M\cong M(S)^{C_2}\cong M(S)^G \cong H_G ^0 (|\cC|; \cM).
\]
Hence the isomorphism in Theorem \ref{thm:Main} holds for the fusion system $\cF$.
Note that in this example, $G$ is the discrete group realizing the fusion system $\cF$ constructed using the Leary-Stancu model.
\end{example}

%----------------------------

\section{Realization of fusion systems and fusion orbit category}\label{sect:FusionAndHigher}

In this section we introduce some basic definitions related to fusion systems and the fusion orbit category of a discrete group.
For further details on this material, we refer the reader to  \cite{AKO-Book} and \cite{Craven-Book}.

\subsection{Fusion systems and their realizations}\label{sect:Fusion} 

\begin{definition}\label{def:FusionSystem}
A \emph{fusion system} $\cF$ over a finite $p$-group $S$ is a category whose objects are subgroups of $S$,
and whose set of morphisms $\Mor _{\cF} (P, Q)$ between two subgroups $P, Q \le S$ satisfies the 
following properties:
\begin{enumerate}
\item $\Hom_S(P, Q) \subseteq  \Mor _{\cF} (P, Q) \subseteq \mathrm{Inj} (P,Q)$ where $\Hom _{S} (P, Q)$ denotes the set of 
all conjugation maps induced by elements of $S$.
\item If $\varphi :  P \to Q$ is a morphism in $\cF$, then $\varphi ^{-1}: \varphi P \to P$ is a morphism in $\cF$.
\end{enumerate}
\end{definition}

Two subgroups $P$ and $Q$ are \emph{$\cF$-conjugate} if there is an isomorphism $\varphi : P \to Q$ in $\cF$. 
In this case we write $P \sim _{\cF } Q$.  A subgroup $P\le S$ is \emph{fully $\cF$-normalized}  if $|N_S (P )| 
\geq |N_S (P') |$ for every $P'\sim _{\cF } P$.  A subgroup $P\le S$ is \emph{fully $\cF$-centralized} 
if $|C_S (P )| \geq |C_S (P') |$ for every $P'\sim _{\cF } P$. 
 
\begin{definition}[{\cite[Definition 1.2]{BLO2}}]\label{def:Saturated} A fusion system $\cF$ over $S$ is \emph{saturated} 
if it satisfies the following properties:
\begin{enumerate} 
\item If $P \le S$ is fully $\cF$-normalized, then $P$ is fully $\cF$-centralized and $\Aut _S (P)$ 
is a Sylow $p$-subgroup of $\Aut _{\cF} (P)$. 
\item If $P \le S$ and $\varphi \in \Mor _{\cF} (P, S)$ such that $\varphi P$ is fully $\cF$-centralized,
then $\varphi$ extends to a morphism $\widetilde \varphi : N_{\varphi } \to S$ where 
$$ N_{\varphi } := \{ x \in N_S (P) \, | \, \exists y \in N_S (\varphi P ) \text{ such that} \ c_x 
= \varphi^{-1} \circ c_y \circ \varphi \}.$$ 
\end{enumerate}
\end{definition}

A finite $p$-subgroup $S$ of a discrete group $G$ is called a Sylow $p$-subgroup of $G$ if every finite $p$-subgroup of $G$ is conjugate to a subgroup of $S$. For a discrete group $G$ with a Sylow $p$-subgroup $S$, the fusion system $\mathcal{F}_S(G)$ is the category whose objects are the subgroups of $S$, and whose morphisms $P \to Q$ are given by conjugation maps 
$c_g: P \to Q$ with $g \in G$.  If $G$ is a finite group, the category $\cF_S(G)$ is a saturated fusion system over $S$ (see \cite[Proposition 1.3]{BLO2}). This is not true in general for an infinite group.

A fusion system $\cF$ over $S$ 
is said to be \emph{realized by $G$}, if $G$ has a Sylow $p$-subgroup 
isomorphic to $S$ and there is an isomorphism of categories $\cF \cong \cF_S(G)$. We have the following realization theorem for fusion systems.
 
\begin{theorem}[Leary-Stancu \cite{Leary-Stancu}, Robinson \cite{Robinson}]\label{thm:Realization} Every fusion 
system $\cF$ is realized by a discrete group $G$ constructed as the fundamental group of 
a graph of groups $(\cG, Y)$ where $Y$ is a finite graph and for all vertices $v\in V$ and edges $e\in E$, the
isotropy subgroups $G_v$ and $G_e$ are finite groups.
\end{theorem}
 
Leary and Stancu \cite{Leary-Stancu} construct $G$ as an iterated HNN-extension and their construction works for any fusion system (not necessarily  saturated). 
In Robinson's model the group $G$ is constructed as a generalized amalgamation of finite groups with edge groups that are finite $p$-groups (see \cite[Theorem 2]{Robinson} and \cite[Theorem 2.1]{Seeliger} for details).

\subsection{Fusion orbit category} Let $G$ be a discrete group, and $\cC$ be a collection of subgroups in $G$. The \emph{fusion category} $\cF_{\cC} (G)$ over $\cC$ is the category whose objects are subgroups in $\cC$ 
and whose morphisms from $P$ to $Q$ are given by $$\Mor _{\cF _{\cC} (G) } (P, Q) = N_G (P, Q) / C_G(P)$$
where $N_G (P, Q)=\{ g \in G \, |\, gPg^{-1} \le Q \}$. The \emph{fusion orbit category} $\overline{\cF}_{\cC} (G)$ over $\cC$ is the category whose objects are subgroups in $\cC$, 
and whose morphisms $P\to Q$ are given by $$\Mor _{\overline{\cF} _{\cC} (G) } (P, Q) = Q\backslash N_G (P, Q) / C_G(P).$$

Let $G$ be a discrete group realizing a fusion system $\cF$.
A $p$-subgroup $P\le G$ is said to be \emph{$p$-centric} in $G$ if $Z(P)$ is a Sylow
$p$-subgroup of $C_G (P)$.  By \cite[Lemma 8.9]{Yalcin-Higher}, a $p$-subgroup $P \le G$ is $G$-conjugate to an 
$\cF$-centric subgroup if and only if it is $p$-centric in $G$. 
As a consequence of this fact, we have the following:

\begin{lemma}[{\cite[Lemma 8.9]{Yalcin-Higher}}]\label{lem:Equivalent}
Let $\cF$ be a saturated fusion system, $G$ be a discrete group realizing $\cF$, and
$\cCe$ denote the collection of all $p$-centric subgroups in $G$.
Then, there are equivalences 
of categories $\cF^c \cong \cF_{\cCe} (G)$ and $\cO ^c (\cF) \cong \overline{\cF} _{\cCe} (G)$.
\end{lemma}

We use these equivalences for the calculation of higher limits over the centric orbit category $\cO ^c (\cF)$. Modules and higher limits over the centric orbit category $\cO^c (\cF)$ are defined using the definitions in Section \ref{Sect:HigherLimits} with $\bC=\cO^c (\cF)$.  An $R\cO^c (\cF)$-module $M$ is called an \emph{atomic functor} concentrated at $Q\in \cF^c$ if for every $\cF$-centric subgroup $P\le S$, we have $M(P)=0$ if $P$ is not $\cF$-conjugate to $Q$. 

For higher limits of an atomic $R\cO ^c (\cF)$-module, we have the following proposition due to Broto, Levi, and Oliver \cite{BLO2}.
 
\begin{proposition}[{\cite[Proposition 3.2]{BLO2}}]\label{pro:HigherFusionOrbit}
Let $\cF$ be a saturated fusion system over a finite $p$-group $S$. Suppose that $M$ is an atomic $R\cO ^c (\cF)$-module concentrated at $Q\in \cF^c$. Then there is an isomorphism
$$\underset{\cO ^c (\cF)}{\lim {}^*} \,  M \cong \Lambda ^* ( \Out _{\cF} (Q), M(Q) )$$ where $\Out _{\cF} (Q) \cong  \Inn (Q) \backslash \Aut _{\cF} (Q)$. 
\end{proposition}  

This proposition is the main ingredient in the proof of Theorem \ref{thm:Main}.

%------------------------

\section{Normalizers of \texorpdfstring{$p$}{p}-centric subgroups}\label{sect:Normalizers}

In this section we prove Proposition \ref{pro:Quotient}, which will be used in the proof of Theorem \ref{thm:Main}. Let $\cF$ be a saturated fusion system over a finite $p$-group $S$, and let $G$ be a discrete group realizing $\cF$. 
Recall that a subgroup $Q \le S$ is fully $\cF$-normalized  if $|N_S (Q )| \geq |N_S (P) |$ for every $P \sim _{\cF} Q$. The following is proved by Libman \cite{Libman-WebbConj}.

\begin{lemma}[{Libman \cite[Proposition 3.8]{Libman-WebbConj}}]\label{lem:Libman}
Let $\cF$ be a saturated fusion system over $S$, and $G$ be a discrete group realizing $\cF$. Then for every fully $\cF$-normalized subgroup $Q \le S$, the normalizer $N_S (Q)$ is a Sylow $p$-subgroup of $N_G (Q)$.
\end{lemma}

% \begin{proof} See the proof of  \cite[Proposition 3.8]{Libman-WebbConj}.
% \end{proof}

For a subgroup $Q\le S$, the normalizer fusion system $N_{\cF} (Q)$ is the fusion system over $N_S(Q)$ where for every $P, P' \le N_S(Q)$, the morphism set $\Mor _{N_{\cF } (Q)} (P, P')$ is the set of all  morphisms $\varphi \in  \Mor _{\cF} (P,P')$ which extend to a morphism $\widetilde \varphi : Q P \to Q P'$ in $\cF$ such that $\widetilde \varphi (Q)=Q$. If $Q$ is fully $\cF$-normalized, then $N_{\cF} (Q)$ is a saturated fusion system by a theorem of Puig (see \cite[Theorem 4.33]{Craven-Book}). Moreover, if $\cF$ is realized by a discrete group $G$, then $N_{\cF} (Q)\cong \cF_{N_S(Q)} (N_G(Q))$ (see \cite[Proposition 8.6]{Yalcin-Higher}). 

As a generalization of these results, we consider the normalizer fusion system for a chain of subgroups of $S$. For a chain of finite subgroups $\sigma =(P_0< \cdots< P_n)$ of $S$, and a subgroup $H\le G$, we define $N_H (\sigma)=\cap _{i=0}^n N_H(P_i)$.

\begin{definition}\label{def:NormFusion}
Let $\sigma =(P_0<\cdots< P_n)$ be a chain of subgroups of $S$. The normalizer fusion system $N_{\cF} (\sigma)$ is the fusion system over $N_S(\sigma)$ where for every $P, Q \le N_S(\sigma)$ the morphism set $\Mor _{N_{\cF } (\sigma)} (P, Q)$ is the set of all  morphisms $\varphi \in  \Mor _{\cF} (P,Q)$ which extend to a morphism
$\widetilde \varphi : P_n P \to P_n Q$ in $\cF$ such that
$\widetilde \varphi ( P_i) = P_i$ for all $0\le i \le n$.
\end{definition}

We say two chains
$\sigma =(P_0< \cdots < P_n )$ and $\tau =(P_0'< \cdots < P_n')$
are $\cF$-conjugate if there is a morphism $\varphi :P_n \to P_n'$ in $\cF$ such that $\varphi (P_i)=P_i'$ for all $i$. In this case we write $\sigma \sim _{\cF} \tau$. We say a chain $\sigma=(P_0<\cdots < P_n)$ in $S$ is fully $\cF$-normalized if $|N_S (\sigma)| \geq |N_S(\tau)|$
for every $\tau \sim _{\cF} \sigma$. We have the following:

\begin{proposition}[{\cite[Proposition 10.3]{Yalcin-Higher}}]  
Let $\cF$ be a saturated fusion system over $S$, and let $\sigma=(P_0<\cdots< P_n)$ be a chain of subgroups of $S$ which is fully $\cF$-normalized. Then,
\begin{enumerate}
\item the normalizer fusion system $N_{\cF} (\sigma)$ is saturated, and
\item if $\cF$ is realized by a discrete group $G$, then $N_S(\sigma)$ is a Sylow
$p$-subgroup of $N_G(\sigma)$ and $\cF _{N_S(\sigma)} (N_G(\sigma))=N_{\cF} (\sigma)$. 
\end{enumerate}
\end{proposition}

Let $G$ be a discrete group realizing a saturated fusion system $\cF$. By the first saturation axiom in Definition \ref{def:Saturated}, if $Q$ is a fully $\cF$-normalized subgroup of $S$, then $\Aut _S (Q)$ is a Sylow $p$-subgroup of $\Aut _{\cF} (Q)$. Since $\Aut _{\cF} (Q)\cong N_G(Q)/C_G(Q)$ and $\Aut _S(Q)\cong N_S(Q)/C_S(Q)$, we can conclude that $C_G(Q)N_S(Q)/C_G(Q) \cong N_S(Q)/C_S(Q)$ is a Sylow $p$-subgroup of $N_G(Q)/C_G(Q)$. Moreover, by  Lemma \ref{lem:Libman}, $N_S(Q)$ is a Sylow $p$-subgroup of $N_G(Q)$. We have the following diagram of subgroups:
\[
\begin{tikzpicture}
\node (G) at (2,4) {$G$};
\node (S) at (2,2) {$S$};
\node (NG) at (1,3) {$N_G(Q)$};
\node (H) at (0,2) {$C_G(Q)N_S(Q)$};
\node (CG) at (-1,1) {$C_G(Q)$}; 
\node (NS) at (1,1) {$N_S(Q)$};
\node (CS) at (0,0) {$C_S(Q)$};

\draw (G) -- (NG);
\draw (G) -- (S);
\draw (S) -- (NS);
\draw (NG) -- (H);
\draw (H) -- (NS);
%\draw (H) -- (CG);
%\draw (NS) -- (CS);
\draw (CG) -- (CS);
\draw[double, double distance=1.5pt] (H) -- (CG);
\draw[double, double distance=1.5pt] (NS) -- (CS);

\end{tikzpicture}
\]

%The following also holds:

\begin{lemma}\label{lem:Extension}
Let $\cF$ be a saturated fusion system over $S$, and $G$ be a discrete group realizing $\cF$. If $Q$ is a fully $\cF$-normalized subgroup of $S$, then $N_S(Q)$ is a Sylow $p$-subgroup of $C_G(Q)N_S(Q)$. 
\end{lemma}

\begin{proof}
Let $H=C_G(Q)N_S(Q)$ and $R$ be a finite $p$-subgroup of $H$. Since $R\le N_G(Q)$, the product $QR$ is a finite $p$-subgroup of $G$. Let $g\in G$ be such that ${}^g (QR) \le S$. Consider the homomorphism 
\[
\varphi=c_{g^{-1}}: {}^g Q \to Q
\]
defined by conjugation with $g^{-1}$. Since ${}^g Q \le {}^g (QR)\le S$, the homomorphism $\varphi$ is a morphism in $\cF$. We claim that ${}^g R\le N_{\varphi}$. 

Since ${}^g R \le {}^g (QR) \le S$ and ${}^g R\le N_G({}^g Q)$, we have ${}^g R \le N_S({}^g Q)$. Let $x\in R$. Then $gxg^{-1}\in N_S({}^g Q)$. Consider the homomorphism $c_{gxg^{-1} } : {}^g Q \to {}^g Q$.
We have 
\[
c_{gxg^{-1}} =\varphi ^{-1} \circ c_x \circ \varphi.
\]
Since $x\in C_G(Q)N_S(Q)$, we have $c_x \in \Aut _S (Q)$.
This implies that $gxg^{-1} \in N_{\varphi}$.
Since this is true for all $x\in R$, we can conclude that ${}^g R \le N_{\varphi}$.
 
Since $Q$ is fully $\cF$-normalized, it is fully $\cF$-centralized, hence by the second saturation axiom in Definition \ref{def:Saturated}, the morphism $\varphi : {}^g Q \to S$ extends to a morphism $\widetilde \varphi : N_{\varphi} \to S$ in $\cF$. 
% By restricting $\widetilde \varphi$ to ${}^g R\le N_{\varphi}$, we obtain a morphism 
% \[
% \widetilde \varphi |_{{}^g R} : {}^g R \to S
% \]
% in $\cF$. 
Let $k\in G$ be such that $c_k=\widetilde \varphi$. 
Then ${}^{kg} R= \widetilde \varphi ({}^g R)\leq S$.

For every $u\in Q$, $$c_{kg} (u) =c_k ( gug^{-1} )=\widetilde{\varphi} (gug^{-1})=  \varphi (gug^{-1} )=c_{g^{-1}} ( gug^{-1})= u.$$
Hence $kg\in C_G (Q)$. Since $R$ normalizes $Q$, $c_{kg} (R)$ normalizes $c_{kg}(Q)=Q$. Thus ${}^{kg} R \le N_G(Q)$. We also have ${}^{kg} R\le S$ by the above argument. Hence ${}^{kg} R \le N_S(Q)$. 

Note that $kg\in C_G (Q) \le H$. Hence we showed that for every finite $p$-subgroup $R\le H$, there is an element $h\in H$ such that ${}^h R \le N_S(Q)$. Therefore $N_S(Q)$ is a Sylow $p$-subgroup of $H=C_G(Q)N_S(Q)$.
\end{proof}

We will also need the following lemma.

\begin{lemma}\label{lem:GroupThy}
    Let $K$ be a normal subgroup of a discrete group $G$ such that $G/K$ is finite.  
    Assume that $G$ has a Sylow $p$-subgroup $S$ such that 
    \begin{itemize}
        \item[(i)] $S$ is a Sylow $p$-subgroup of $SK$,
        \item[(ii)] $SK/K$ is a Sylow $p$-subgroup of $G/K$, and
        \item[(iii)] $S \cap K=1$.
    \end{itemize}
    Then there is a $G/K$-equivariant simplicial isomorphism
$S_p(G)/K \cong S_p(G/K)$,
and hence a $G/K$-equivariant homeomorphism
\[
|S_p(G)|/K \cong |S_p(G/K)|,
\]
where $S_p(G)$ denotes the simplicial complex of strict chains of nontrivial finite $p$-subgroups in $G$.
\end{lemma}

\begin{proof} The proof is similar to the proof of \cite[Proposition 5.7]{Grodal-Higher}, modified so that the statement still holds for infinite groups. Let $G$, $K$, and $S$ be as in the lemma. Since the conditions given in the lemma are invariant under conjugation, if $T$ is another Sylow $p$-subgroup of $G$, then $T$ also satisfies the conditions of the lemma.

For a finite $p$-subgroup $P\le G$, let $\overline P=PK/K$, and write $\overline G=G/K$. For every Sylow $p$-subgroup $T$ of $G$, we have 
\[
\overline T \cong TK/K \cong T/(T\cap K) \cong T
\]
since $T\cap K=1$. This isomorphism sends a subgroup $P\le T$ to the quotient group $\overline P\le \overline T$. Hence the assignment $P\to \overline P$ gives a bijection between the set of $p$-subgroups of $T$ and the set of $p$-subgroups of $\overline T$.

Every finite chain of $p$-subgroups is contained in a Sylow $p$-subgroup of $G$. Hence the assignment $(P_0<\cdots <P_n) \mapsto (\overline P_0< \cdots < \overline P_n)$ sends a strict chain of nontrivial $p$-subgroups of $G$ to a strict chain of nontrivial $p$-subgroups of $\overline G$. This defines a simplicial map 
%$f': S_p(G) \to S_p(G/K)$. If two chains in $S_p(G)$ are equivalent under $K$-conjugation, then their images under $f'$ are equal in $S_p (G/K)$. Hence $f'$ induces a simplicial map 
\[
f: S_p(G)/K \to S_p (G/K)
\]
that sends a $K$-orbit $[P_0 < \cdots<P_n]_K$ of a chain to $(\overline P_0 < \cdots < \overline P_n)$ in $S_p (G/K)$. 
 
We now show that $f$ is surjective. Let $(R_0 < \cdots < R_n)$ be a 
chain of nontrivial $p$-subgroups in $G/K$. Let $S$ be a Sylow $p$-subgroup of $G$. Since $\overline S$ is a Sylow $p$-subgroup of $\overline G$, there is an element $g\in G$ such that ${}^{gK} R_n \le \overline S$. By replacing $S$ with its conjugate ${}^{g^{-1}} S$, we can assume that $R_n \le \overline{S}$. Since the quotient map restricts to an isomorphism \(S\to\overline S\), define
\[
P_i=\{x\in S\mid xK\in R_i\}
\qquad (0\le i\le n).
\]
Then $(P_0<\cdots<P_n)$ is a chain of nontrivial $p$-subgroups of $G$ with $\overline{P_i}=R_i$ for every $i$. Hence $f$ is surjective.

% Let $P_n=\{ x\in S \, |\, xK\in R_n \}$. Then $P_n$ is a finite $p$-subgroup of $G$ such that $\overline P_n=R_n$. Suppose that there is a chain $P_{i} < \dots < P_n$ of subgroups of $S$ for some $0<i\le n$ such that $\overline P_j=R_j$ for all $i\le j \le n$. Then we can define $P_{i-1}$ to be the subgroup $\{ x \in P_{i} \, | \, xK \in R_{i-1} \}$. By induction, we can conclude that there is a chain of nontrivial $p$-subgroups $(P_0 < \dots < P_n)$ in $G$ such that $(\overline P_0 < \cdots < \overline P_n )=(R_0 < \cdots < R_n)$ in $\overline G$. Hence $f$ is surjective.

Finally we show that $f$ is injective. Suppose that $(P_0< \dots < P_n)$ and $(P'_0 < \dots < P'_n)$ are two chains in $S_p(G)$ such that 
$\overline P_i=\overline P_i'$ for all $i$. Let $S$ be a Sylow $p$-subgroup of $G$ such that $P_n\le S$. Since $\overline P_n=\overline P_n' $, we have $\overline P_n' \le \overline S$. Hence 
\[
P_n' \le P_n' K \le SK.
\]
Since $S$ is a Sylow $p$-subgroup of $SK$, there is an element $g=sk\in SK$
such that ${}^g (P_n')\le S$. Then
\[
{}^k(P_n')= {}^{s^{-1} g} (P'_n) \le  {}^{s^{-1}}S=S.
\]
Hence, $({}^k (P_0') <\cdots < {}^k (P_n'))$ is a chain  in $S$. Now by the injectivity of the mapping $P \to \overline P$ on subgroups of $S$, we can conclude that $P_i={}^k (P_i')$ for all $i$. Thus $[(P_0<\cdots <P_n)]_K=[(P'_0<\cdots < P_n')]_K$. Hence $f$ is injective.

The simplicial map $f$ respects the $G/K$-actions defined by conjugation. Hence $f$ is a $G/K$-equivariant simplicial isomorphism. Passing to geometric realizations gives a $G/K$-equivariant homeomorphism.
\end{proof}

Combining the previous two lemmas, we obtain the following:

\begin{proposition}\label{pro:Quotient} Let $\cF$ be a saturated fusion system over $S$, and $G$ be a discrete group realizing $\cF$. Suppose that $Q\le S$ is an $\cF$-centric, fully $\cF$-normalized subgroup of $S$. Set $W=N_G(Q)/Q$ and $K=QC_G (Q)/Q$. Then there is a $W/K$-equivariant homeomorphism $|S_p (W/K)| \cong |S_p (W)|/K$.
\end{proposition}

\begin{proof} Let $W$ and $K$ be as above, and $T=N_S(Q)/Q$. Since $Q$ is $\cF$-centric, we have $C_S(Q)=Z(Q)$. We have the following diagrams of subgroups where some of the squares are justified in the following argument:
\[
\begin{tikzpicture}

\node (NG) at (2, 4) {$N_G(Q)$} ;
\node (H) at (1,3) {$N_S(Q)C_G(Q)$};
\node (NS) at (2,2) {$N_S(Q)$};
\node (CGQQ) at (0,2) {$Q C_G(Q)$};
\node (CG) at (-1,1) {$C_G(Q)$}; 
\node (Q) at (1,1) {$Q$};
\node (CS) at (0,0) {$C_S(Q)=Z(Q)$};

\node (W) at (7,4)  {$W$};
\node (KT) at (6,3) {$TK$};
\node (T) at (7,2) {$T$};
\node (K) at (5,2) {$K$};
\node (1) at (6,1) {$1$};

\draw (W) -- (KT);
\draw (KT) -- (T);
\draw (K) -- (1);
\draw[double, double distance=1.5pt] (KT) -- (K);
\draw[double, double distance=1.5pt] (T) -- (1);

\draw (NG) -- (H);
\draw[double, double distance=1.5pt] (NS) -- (Q);
\draw[double, double distance=1.5pt] (H) -- (CGQQ);
\draw (H) -- (NS);
\draw (CGQQ) -- (Q) ;
\draw (CG) -- (CS);
\draw[double, double distance=1.5pt] (CGQQ) -- (CG);
\draw[double, double distance=1.5pt] (Q) -- (CS);
\end{tikzpicture}
\]

The product $QC_G(Q)$ is a normal subgroup of $N_G(Q)$, hence $K$ is a normal subgroup of $W$. The quotient group
\[
W/K \cong N_G(Q)/Q C_G(Q)
\]
is isomorphic to $\Out _{\cF} (Q)$, hence is a finite group. 

By assumption, $Q$ is a fully $\cF$-normalized subgroup of $S$, hence by  Lemma \ref{lem:Libman}, $N_S(Q)$ is a Sylow $p$-subgroup of $N_G(Q)$. Therefore $T$ is a Sylow $p$-subgroup of $W$. Moreover, by Lemma \ref{lem:Extension}, $N_S(Q)$ is a Sylow $p$-subgroup of $N_S(Q)C_G(Q)$. Hence $T$ is a Sylow $p$-subgroup of $TK$, since $TK=N_S(Q)QC_G(Q)/Q=N_S(Q)C_G(Q)/Q$.  

Now consider the quotient group $TK/K$. Since $Q$ is fully $\cF$-normalized, $\Aut _S(Q) \cong N_S(Q)/C_S(Q) \cong N_S(Q)C_G(Q)/C_G(Q)$ is a Sylow $p$-subgroup of $\Aut _{\cF}(Q)\cong N_G (Q)/C_G(Q)$. Hence the finite index $|N_G(Q):N_S(Q)C_G(Q)|$ is coprime to $p$.
This proves that
\[
TK/K \cong N_S(Q)C_G(Q)/QC_G(Q)
\]
is a Sylow $p$-subgroup of $W/K\cong N_G(Q)/QC_G(Q)$.

Finally, we claim that $T\cap K=1$. Since $C_S(Q)=Z(Q)$, we have  $N_S(Q)\cap C_G(Q)=Z(Q)$. Let $x=qc\in QC_G(Q) \cap N_S(Q)$, where $q\in Q $ and $c\in C_G(Q)$. Then 
\[
c=q^{-1}x\in C_G(Q) \cap N_S(Q)=Z(Q) \le Q.
\]
Hence $x=qc \in Q$. This gives that $QC_G(Q)\cap N_S(Q)=Q$. Thus
$T\cap K=1$.  

We verified all the assumptions of Lemma \ref{lem:GroupThy} for $W$ with normal subgroup $K$, and a Sylow $p$-subgroup $T$. Hence by Lemma \ref{lem:GroupThy}, there is a $W/K$-equivariant homeomorphism $|S_p(W/K)| \cong |S_p(W)|/K$.
\end{proof}

%-----------------

\section{Proof of Theorem \ref{thm:Main}}\label{sect:ProofMain}

Let $\cF$ be a saturated fusion system over $S$, and $G$ be a discrete group realizing $\cF$. Let $\cC$ denote the collection of $p$-centric subgroups in $G$. Consider the functor 
\[
\overline \xi: \Delta (|\cC|)_G \to  \cO ^c (\cF)^{op}
\]
defined as the composition of $\xi: \Delta (|\cC|)_G \to  \cO_{\cC} (G) ^{op}$, introduced in Definition \ref{def:xi}, with the quotient functor $\cO _{\cC} (G)^{op} \to \overline \cF_{\cC} (G) ^{op} \simeq \cO ^c (\cF) ^{op}$. 

Let $R=\bbZ_{(p)}$ and $M$ be an atomic $R\cO^c (\cF)$-module concentrated at $Q\in \cF^c$. Suppose that $M'$ is the associated $R\cO_{\cC} (G)$-module obtained by precomposing $M$ with the functor $\cO_{\cC} (G) \to \overline \cF _{\cC} (G) \cong \cO^c (\cF)$, and $\cM$ is the equivariant coefficient system on $|\cC|$ defined by precomposing $M'$ with $\xi : \Delta (|\cC|)_G \to \cO _{\cC} (G)^{op}$. Let $W=N_G(Q)/Q$ and $K=QC_G(Q)/Q$. As before, $S_p^e (W)$ denotes the poset of all $p$-subgroups of $W$, and $S_p (W)$ denotes the poset of all nontrivial $p$-subgroups of $W$.

\begin{lemma}\label{lem:Part1}
With the definitions above, there is an isomorphism 
\[
H^* _G (|\cC|; \cM )\cong H^* _{W/K} (|S^e _p (W)|/K, |S_p (W)|/K; M(Q) ).
\]

\end{lemma}

\begin{proof}
%Let $M$ be an atomic $\bbZ_{(p)}\cO^c(\cF)$-module concentrated at $Q \in \cF^c$. Since $Q$ is defined up to $\cF$-conjugation, we can assume that $Q$ is fully $\cF$-normalized. 
%Let $M'$ be the $\bbZ_{(p)}\cO_{\cC} (G)$-module defined by precomposing $M$ with the quotient functor $\cO_{\cC} (G) \to \cO ^c (\cF)$. Then $M'$ is also an atomic functor concentrated at $Q$.
%Let $M$ be an atomic $R\cO^c (\cF)$-module concentrated at $\cF$-centric subgroup $Q$. 
For each $n\geq 0$, the normalized cochain complex has the form 
\begin{align*}
     C^n _G (|\cC| ; \cM ) & =\Bigl (\prod _{(H_0 < \cdots < H_n ) \in |\cC|_n} M' (H_0) \Bigr )^G  \cong \Bigl (\prod _{(Q=H_0 <  \cdots < H_n ) \in |\cC|_n} M(Q) \Bigr )^{N_G(Q)} 
     \\
& \cong \Hom _{RN_G(Q) } \bigl ( C_n(|\cC _{\geq Q} |, |\cC _{>Q}| ), M(Q) \bigr ) \cong \Hom _{RW } \bigl ( C_n(|\cC _{\geq Q} |, |\cC _{>Q}| ), M(Q) \bigr ). 
\end{align*}
Since all the isomorphisms above commute with coboundary maps, this gives an isomorphism of cochain complexes, hence we have the isomorphism of the corresponding cohomology groups
%Applying Lemma \ref{lem:EquivCoh} to $R\cO_{\cC} (G)$-module $M'$, we have an isomorphism
\[
H^* _G (|\cC| ; \cM) \cong H^* _{W} (|\cC _{\geq Q} |, |\cC _{>Q}| ; M(Q)).
\]
%where $W=N_G(Q)/Q$. 

Since the collection of $p$-centric subgroups is closed under passage to $p$-overgroups, $\cC _{\geq Q} = S^e_p (G) _{\geq Q}$ and $\cC_{>Q} =S_p(G)_{>Q}$. There is an $N_G(Q)$-equivariant homotopy equivalence 
\[
|S_p ^e (G) _{\geq Q}| \simeq |S^e_p (N_G(Q) )_{\geq Q}|
\]
given by the poset homotopy induced by the poset map $P \to N_P(Q)$ (see \cite[Lemma 3.3.18]{Smith-2011}). Passing to $Q$ quotients, this gives a $W$-equivariant homotopy equivalence $|S^e_p (G) _{\geq Q}| \simeq |S_p ^e (W)|$. For every $P>Q$, we have $N_P(Q)>Q$, hence this homotopy equivalence restricts to a
$W$-equivariant homotopy equivalence  $|S_p(G)_{>Q}|\simeq |S_p (W)|$, hence we obtain a $W$-equivariant homotopy equivalence of pairs $(|\cC_{\geq Q}|, |\cC _{>Q}| )\simeq (|S^e_p (W)|, |S_p (W) |)$. 
This gives an isomorphism of cohomology groups
$$H^* _G (|\cC| ; \cM) \cong H^* _{W} (|S_p ^e (W) |, |S_p (W) | ; M(Q) ).$$ 
Since $C_G(Q)$ acts trivially on $M(Q)$, by taking $K$-orbits, we obtain
\begin{equation*}\label{eqn:Part1}
H^* _G (|\cC| ; \cM) \cong H^* _{W/K} (|S_p ^e (W)|/K , |S_p (W) |/K ; M(Q)). 
\end{equation*}
%where $K=QC_G(Q)/Q$. %Note that $W/K \cong N_G(Q)/QC_G(Q) \cong \Out _{\cF} (Q) =\Gamma$.

The isomorphism above is induced by the cochain map defined on unnormalized cochains as
\[
\zeta^* : C_G ^* (|\cC|; \cM) \to C^* _{W/K} (|S_p^e (W)|/K, |S_p (W)|/K; M(Q))
\]
defined by 
\[
\zeta^* (f)([1=U_0 \le \cdots \le U_n]_K)=f(Q=Q_0\le \cdots \le Q_n)
\]
for every $f\in C_G^* (|\cC| ; \cM)$, and every $[1=U_0\le \cdots \le U_n]_K \in |S_p ^e (W)|/K$, where the subgroups $Q_i$ are defined such that $Q\le Q_i \le N_G(Q)$ and $Q_i/Q= U_i$ for all $i$.
\end{proof}

For every  $R\cO ^c \cF$-module $M$, there is an isomorphism $\lim ^* _{\cO ^c (\cF)} M \cong H^* (\cO^c (\cF); M)$. The cohomology groups $H^* (\cO ^c (\cF) ; M)$ can be calculated using the standard projective resolution $P_*$ for the constant functor $\underline R$ as an $R\cO ^c (\cF)$-module. This allows us to express $H^* (\cO ^c (\cF); M)$ as the cohomology of the cochain complex $C^* (\cO ^c (\cF); M)$ where for each $n\geq 0$, 
\[
C^n (\cO ^c (\cF) ; M) = \prod _{(P_0 \maprt{\alpha_1}  \cdots \maprt{\alpha_n} P_n)\in N_n (\cO ^c (\cF))} M(P_0).
\]

Now we will define a cochain map
\[
\Psi^* : C^* _G (|\cC|; \cM)  \to C^* (\cO ^c (\cF); M).
\]
For this definition we take $C_G^* (|\cC|; \cM)$ as the unnormalized cochain complex where for each $n\geq 0$,
\[
C_G ^n (|\cC|; \cM)=\Bigl (\prod _{(Q_0\le\cdots \le Q_n) \in N_n(\cC)} M(Q_0) \Bigr ) ^G.
\]
Let $(P_0 \maprt{\alpha_1 } \cdots \maprt{\alpha_n} P_n )$ be a simplex in $N_n (\cO ^c (\cF))$. For each $i=1,\dots, n$, choose a morphism $\varphi_i : P_{i-1} \to P_i$ in $\cF$ such that $\alpha_i=[\varphi_i]$. For each $j=0, \dots , n-1$, define $\theta _j= \varphi_n \circ \cdots \circ \varphi_{j+1}$, and set $\theta_n=\id$. Let $Q_i=\theta _i (P_i)$ for all $i$. Viewing $\theta_0$ as an isomorphism $P_0\to Q_0$, we define $\Psi^n$ by setting
\[
\Psi ^n (f) (P_0 \maprt{\alpha_1} \cdots \maprt{\alpha_n} P_n)=M(\theta_0) f (Q_0\le \dots\le Q_n)
\]
for each $f\in C^n_G (|\cC|; \cM)$. Since $f$ is $G$-invariant, the definition of $\Psi^n (f)$ does not depend on different choices of representative morphisms $\varphi_i$ for $\alpha_i$.
From its definition, it is easy to see that $\Psi^*$ defines a cochain map. Therefore, it induces a homomorphism on cohomology
\[
\Psi^* : H^* _G (|\cC|; \cM) \to H^* (\cO ^c (\cF); M).
\]

For the cohomology groups $H^* (\cO^c (\cF), M)$ with coefficients in an atomic functor, we have the following observation.

\begin{lemma}\label{lem:Part2}
Let $\cF$ be a saturated fusion system over $S$, and $Q$ be an $\cF$-centric, fully $\cF$-normalized subgroup of $S$. Let $\Gamma=\Out_{\cF} (Q)$, and $M$ be an atomic $\bbZ_{(p)} \cO ^c (\cF)$-module concentrated at $Q$. Then there is an isomorphism
\begin{equation*}\label{eqn:Part2}
H^* (\cO ^c (\cF); M) \cong 
H^* _{\Gamma}  ( |S_p^e (\Gamma)| , |S_p (\Gamma)| ; M (Q)). 
\end{equation*} 
\end{lemma}

\begin{proof}
Let $E_p \Gamma$ denote the universal space for $\Gamma$ over all $p$-subgroups of $\Gamma$. For $E_p\Gamma$ we will use the simplicial model given by the nerve of the category $\cE _p (\Gamma)$, whose isotropy subgroups are $p$-subgroups of $\Gamma$. 
By Propositions \ref{pro:HigherFusionOrbit} and \ref{pro:LimBredon}, we have 
$$H^* (\cO ^c (\cF) ; M) \cong \underset{\cO ^c (\cF) }{\lim {}^*} M  \cong \Lambda ^* (\Gamma,  M (Q) ) \cong \underline{H} ^*_{\Gamma} (E_p \Gamma ^?; F_{M(Q)} ).$$
where $F_{M(Q)}$ denotes the atomic functor concentrated at $1$ with value $M(Q)$. The Bredon cohomology on the right is isomorphic to the equivariant cohomology
\[
H_{\Gamma}  ^*  ( E_p \Gamma , (E_p \Gamma)^s ; M (Q))\]
where $(E_p \Gamma )^s =\bigcup _{1\neq P \le \Gamma} (E_{p} \Gamma)^P$ is the singular set of the $\Gamma$-action on $E_{p} \Gamma$.

Since $\Gamma$ is a finite group and $M(Q)$ is a $\bbZ _{(p)} \Gamma$-module, we can apply Proposition \ref{pro:CohIsom}, and obtain an isomorphism
$$ H ^* _{\Gamma}  (E_p \Gamma , (E_p \Gamma)^s ; M(Q)) \cong H^* _{\Gamma}  ( |S_p^e (\Gamma)|, |S_p (\Gamma)| ; M(Q)).$$ 
This isomorphism is induced by the functor $\cE _p \Gamma \to S_p ^e (\Gamma)$ that sends $(\Gamma/P, gP)$ to ${}^g P$.
Combining the isomorphisms above, we obtain that
\begin{equation*}\label{eqn:Part2}
H^* (\cO ^c (\cF); M) \cong 
H^* _{\Gamma}  ( |S_p^e (\Gamma)| , |S_p (\Gamma)| ; M (Q)). 
\end{equation*}

From the descriptions above, we see that the isomorphism above is induced by a zigzag of cochain maps.
To describe these cochain maps, observe that the isomorphism given in Proposition \ref{pro:HigherFusionOrbit} is constructed using the explicit functor $\alpha: \cO _{\Gamma_p} (\Gamma) \to \cO ^c (\cF)$, where $\Gamma _p=\Out _S (Q)$ is a Sylow $p$-subgroup of $\Gamma$, and $\cO _{\Gamma_p} (\Gamma)$ is the orbit category with objects subgroups of $\Gamma_p$. This functor sends $V\le \Gamma_p$ to 
\[
\widehat V=\{ x\in N_S(Q) : [c_x|_Q] \in V \} \le N_S(Q)
\]
(see \cite[Proposition 3.2]{BLO2}). %Note that since $Q$ is $\cF$-centric, $\widehat V$ is $\cF$-centric. 
The functor $\alpha$ induces a cochain map $\alpha^* : C^* (\cO^c (\cF); M) \to C^* (\cO _{\Gamma_p} (\Gamma); F_{M(Q)})$. 

Now consider the cochain map $\eta ^* : C^* _{\Gamma}  ( |S_p^e (\Gamma)| , |S_p (\Gamma)| ; M (Q)) \to C^* ( \cO _{\Gamma_p} (\Gamma); F_{M(Q)})$ defined by 
\[
\eta^*(f) (V_0 \maprt{\beta_1} \cdots \to V_{n-1} \maprt{\beta_n} V_n)= M (\gamma_n\cdots \gamma_1)f({}^{\gamma_n\cdots \gamma_1} V_0\le \cdots \le  {}^{\gamma_n} V_{n-1} \le V_n)
\]
where for each $i$, the morphism $\beta_i$ is given by the $\Gamma$-map $\Gamma/V_{i-1} \to \Gamma/V_i$ defined by $V_{i-1} \to \gamma_i ^{-1} V_i$, for all $i$, and for  $\gamma \in \Gamma$, $M(\gamma): M(Q)\to M(Q)$ is the $\bbZ_{(p)}$-module homomorphism induced by the morphism $\gamma: Q \to Q$ in $\Aut_{\cO^c (\cF)} (Q)\cong \Out _{\cF}(Q)=\Gamma$. 

The orbit-category cochains in $C^*(\cO_p (\Gamma); F_{M(Q)})$ can be identified with the equivariant cochains on $E_p \Gamma$ relative to its singular set. With this identification, the map $\eta^*$ is the pullback induced by the map in Proposition \ref{pro:CohIsom}, followed by restriction to $\cO_{\Gamma_p} (\Gamma)$. Hence by Proposition \ref{pro:CohIsom} and the equivalence $\cO_{\Gamma_p} (\Gamma)\simeq \cO _p (\Gamma)$, we can conclude that $\eta^*$ induces an isomorphism on cohomology. The isomorphism stated in the lemma
% \[
% H^* (\cO ^c (\cF); M) \cong 
% H^* _{\Gamma}  ( |S_p^e (\Gamma)| , |S_p (\Gamma)| ; M (Q))
% \]
is equal to the composition $H(\eta^*)^{-1} \circ H(\alpha^*)$ where $H(\alpha^*)$ and $H(\eta^*)$ are the induced homomorphisms on cohomology.
\end{proof}

%Our main result in this section is the following:

Now we prove our key result in this section.

\begin{proposition}\label{pro:MainAtomic} If $M$ is an atomic $\bbZ _{(p)} \cO ^c (\cF)$-module concentrated at an $\cF$-centric subgroup $Q \le S$, and $\cM$ is the equivariant coefficient system on $|\cC|$ associated to $M$, then the homomorphism $\Psi^*$ defined above, is an isomorphism.
\end{proposition}

\begin{proof} Replacing $Q$ by an $\cF$-conjugate subgroup, we may assume that $Q$ is fully $\cF$-normalized.
Let $W=N_G(Q)/Q$, $K=QC_G(Q)/Q$, and $\Gamma=\Out_{\cF} (Q)\cong W/K$. 
Consider the following diagram
\begin{equation}\label{eqn:CommDiagram}
\vcenter{\hbox{
\xymatrix{ H_G ^* (|\cC|; \cM) \ar[r]^-{\xi _1 \cong} \ar[d]^{\Psi^*} & H^* _{\Gamma} (|S_p ^e (W)|/K , |S_p (W) |/K ; M(Q) )  \ar[d]^{\kappa^*} \\
H^* (\cO ^c (\cF) ; M) \ar[r]^-{\xi _2 \cong}& H^* _{\Gamma}  ( |S_p^e (\Gamma)| , |S_p (\Gamma)| ; M (Q)) }
}}
\end{equation}
where the horizontal maps $\xi_1$ and $\xi_2$ are the isomorphisms proved in Lemmas \ref{lem:Part1} and \ref{lem:Part2}.
The right vertical map $\kappa ^*$ is the homomorphism defined as follows:  By Proposition \ref{pro:Quotient}, there is a $\Gamma$-equivariant homeomorphism $f: |S_p(W)|/K \to |S_p(\Gamma)|$. This homeomorphism 
extends to the cones $|S_p ^e (W)|/K$ and $|S^e _p (\Gamma)|$, and gives a $\Gamma$-equivariant homeomorphism of pairs
\[
 \widetilde f: (|S_p ^e (W) |/K , |S_p (W)|/K ) \to (|S_p^e (\Gamma)| , |S_p (\Gamma)|). 
\]
The inverse $(\widetilde f)^{-1}$ of this homeomorphism defines a cochain map
\[
\kappa ^*: C^* _{\Gamma} (|S_p ^e (W)|/K, |S_p (W)|/K; M(Q)) \to  C^* _{\Gamma} (|S^e _p (\Gamma)|, |S_p (\Gamma)|; M(Q))
\]
which is an isomorphism of cochain complexes. In particular, $\kappa^*$ induces an isomorphism of corresponding cohomology groups.

If the diagram in \eqref{eqn:CommDiagram} commutes, then the isomorphisms established above will imply that  $\Psi^*$ is also an isomorphism. To prove the commutativity of the above diagram, we consider corresponding maps on the cochain level:
\[
\xymatrix{C^* _G (|\cC|; \cM) \ar[d]^{\Psi^*} \ar[rr]^{\zeta ^*} & & 
C^* _{\Gamma} (|S_p ^e (W)|/K , |S_p (W) |/K ; M(Q) )  
\ar[d]^{\kappa^*} \\
C^* (\cO ^c (\cF) ; M) \ar[r]^-{\alpha^*} & C^* (\cO _{\Gamma_p} (\Gamma) ; F_{M(Q)}) &  C^* _{\Gamma}  ( |S_p^e (\Gamma)| , |S_p (\Gamma)| ; M (Q)) 
\ar[l]_-{\eta^*}}
\]
where $\zeta^*$, $\alpha^*$, and $\eta^*$ are the cochain maps defined in the proofs of Lemmas \ref{lem:Part1} and \ref{lem:Part2}.
% Then the composition $\kappa^* \circ \xi_1 : C_G ^* (|\cC|; \cM ) \to E^*$
% is the cochain map defined by
% \[
% (\kappa ^* \circ \xi_1 ) (f) (1=V_0\le \cdots \le V_n\le \Gamma)=f(Q=Q_0\le \cdots \le Q_n \le N_G(Q))
% \]
% for every chain $(1=V_0\le \cdots \le V_n)$ in $|S_p^e (\Gamma)|$, where the subgroups $Q_i$ are such that $V_i =C_G(Q) Q_i/ C_G(Q) Q$ for all $i$.

%The commutativity of the above diagram can be verified by direct calculation. 
We will show that $\eta^* \circ \kappa^* \circ \zeta^*=\alpha^* \circ \Psi^*$. Let $f\in C^* _G (|\cC|; \cM)$, and $(V_0 \maprt{\beta_1} \cdots \maprt{\beta_n } V_n)$ be a chain in $\cO _{\Gamma_p} (\Gamma)$ with $V_0=1$. For each $i$, choose $\gamma_i\in\Gamma$ such that $\beta_i(V_{i-1})=\gamma_i ^{-1} V_i$. Let
$\varphi_i:\widehat V_{i-1}\to\widehat V_i$
be a representative of $\widehat\beta_i=\alpha(\beta_i)$ such that
$\varphi_i(Q)=Q$ and $[\varphi_i|_Q]=\gamma_i$.
We have 
\begin{align*}
(\alpha^* \circ \Psi ^* ) (f) (1=V_0 \maprt{\beta_1} \cdots \maprt{\beta_n } V_n) & =
\Psi ^* (f) (Q=\widehat V_0 \maprt{\widehat \beta_1} \cdots \maprt{\widehat \beta_n } \widehat V_n) \\
& = M(\theta_0) f(Q=R_0\le \cdots \le R_n)
\end{align*}
where $\theta _j=\varphi _n \circ \cdots \circ \varphi _{j+1}$ for $j=0,\dots, n-1$, $\theta _n=\id$, and $R_j=\theta _j (\widehat V_j)$ for $j=0, \dots, n$. From the definitions of subgroups $R_j$ and elements $\gamma_j$, we have
\[
QC_G(Q)R_j/QC_G(Q)
=
{}^{\gamma_n\cdots \gamma_{j+1}}V_j
\]
and $M(\theta_0|_Q)=M(\gamma_n\cdots \gamma_1)$. This gives
\begin{align*}
& (\eta^* \circ \kappa ^* \circ \zeta^*) (f) (1=V_0 \maprt{\beta_1} \cdots \maprt{\beta_n } V_n) \\
&=M(\gamma_n\dots \gamma_1) (\kappa^*\circ \zeta^*)(f) (1={}^{\gamma_n\dots \gamma_1}V_0 \le \cdots \le {}^{\gamma_n } V_{n-1} \le V_n ) \\
& =M(\gamma_n\cdots\gamma_1)f(Q=R_0\le\cdots\le R_n).
\end{align*}
Thus both calculations give the same value. For chains with $V_0\neq 1$, both composites have the value zero. Thus we can conclude that $(\alpha^* \circ \Psi^*)(f)=(\eta^* \circ \kappa^* \circ \zeta^*) (f)$. This gives that the induced maps on cohomology satisfy
$\kappa^*\circ\xi_1=\xi_2\circ\Psi^*$,
so the diagram in \eqref{eqn:CommDiagram} commutes.
This completes the proof as explained above.
\end{proof}

Now we are ready to prove the main result of the paper.

\begin{proof}[Proof of Theorem \ref{thm:Main}]
Let $\cF$, $G$, and $\cC$ be as in Theorem \ref{thm:Main}. 
The statement is trivial for $M=0$ so assume $M\neq 0$.
By Proposition \ref{pro:MainAtomic}, the isomorphism in Theorem \ref{thm:Main} holds when $M$ is an atomic functor. We prove Theorem \ref{thm:Main} for an arbitrary $\bbZ _{(p)} \cO^c (\cF)$-module $M$ by induction on the number of $\cF$-conjugacy classes $[P]$ of $\cF$-centric subgroups of $S$
such that $M(P)\neq 0$.  Let $P$ be an $\cF$-centric subgroup of $S$ with maximal order such that $M(P)\neq 0$. 

Let $M'$ be the submodule of $M$ such that $M'(Q)=M(Q)$ for $Q\not \sim_{\cF} P$
and $M'(Q)=0$ if $Q\sim _{\cF} P$. There is a short exact sequence of $\bbZ_{(p)} \cO^c (\cF)$-modules
\[
0 \to M' \to M \to M/M' \to 0 
\]
where $M/M'$ is an atomic functor concentrated at $P$.
This short exact sequence induces a commuting diagram of cochain complexes
\[
\xymatrix{0 \ar[r] & C^* _G (|\cC|; \cM' ) \ar[r]\ar[d]^{\Psi^*_{M'}} & C^* _G (|\cC|; \cM ) \ar[r]\ar[d]^{\Psi^*_M}& C^* _G (|\cC| ; \cM/\cM' ) \ar[r] \ar[d]^{\Psi^* _{M/M'}} & 0 \\
0 \ar[r] & C^* (\cO ^c (\cF); M' ) \ar[r] & C^*  (\cO ^c (\cF); M ) \ar[r] & C^* ( \cO ^c (\cF) ; M/M' ) \ar[r]  & 0 }
\]
where $\cM$ and $\cM'$ are associated equivariant coefficient systems to $M$ and $M'$. 

Recall that $C^* _G (|\cC|; \cM)=C^* (|\cC|; \cM)^G$, and for each $n\ge 0$, the sequence 
\[
\xymatrix{0 \ar[r] & C^n  (|\cC|; \cM' ) \ar[r]  & C^n  (|\cC|; \cM ) \ar[r] & C^n  (|\cC| ; \cM/\cM' ) \ar[r]  & 0}
\]
has a $G$-equivariant splitting, because on each simplex the coefficient sequence is either an identity map followed by zero, or zero followed by an identity map, and this distinction is preserved by the $G$-action.
Hence after taking $G$-fixed points, it is still exact. We conclude that the first horizontal sequence in the above diagram is exact. The exactness of the second sequence follows from the fact that $C^* (\cO^c (\cF); M)=\Hom _{\bbZ_{(p)} \cO ^c (\cF)} (P_* ; M)$ where $P_*$ is a projective resolution of the constant functor.
By the induction hypothesis, $\Psi^*_{M'}$ induces an isomorphism on cohomology. By Proposition \ref{pro:MainAtomic}, the cochain map 
$\Psi^* _{M/M'}$ induces an isomorphism on cohomology. Hence, by the five-lemma, $\Psi^*_M$ also induces an isomorphism on cohomology.
\end{proof}

Let $\cF$ be a saturated fusion system over $S$, and $\cC$ be a collection of subgroups of $S$ closed under $\cF$-conjugacy.  For every strict chain of subgroups $\sigma=(P_0 < \cdots < P_n)$ in $S$, we define $\Aut _{\cF} (\sigma)=\{ \varphi\in \Aut_{\cF} (P_n) \, |\, \varphi (P_i)=P_i \text{ for all } i\}$. We say two chains
$\sigma =(P_0< \cdots < P_n )$ and $\tau =(P_0'< \cdots < P_n')$
are $\cF$-conjugate if there is a morphism $\varphi :P_n \to P_n'$ in $\cF$ such that $\varphi (P_i)=P_i'$ for all $i$. We denote the $\cF$-conjugacy classes of chains in $\cC$ with length $n$ by $|\cC|_n/\cF$. 

\begin{lemma}\label{lem:IsomQuotients}
Let $\cF$ be a saturated fusion system over a finite $p$-group $S$, and $\cC$ be the collection of all $\cF$-centric subgroups in $S$.
Suppose that $G$ is a discrete group with Sylow $p$-subgroup $S$ realizing $\cF$, and $\cC'$ be the collection of all $p$-centric subgroups of $G$. Then for each $n\geq 0$, there is a bijection 
between $|\cC'|_n/G$ and $|\cC|_n/\cF$.
\end{lemma}

\begin{proof}
By \cite[Lemma 8.9]{Yalcin-Higher}, a $p$-subgroup $P \le G$ is conjugate to an $\cF$-centric subgroup
in $S$ if and only if it is a $p$-centric subgroup of $G$.
This implies that every chain of subgroups in $\cC'$ is $G$-conjugate to a chain in $\cC$. Since $\cF$ is realized by $G$, two chains of subgroups in $\cC$ are $G$-conjugate if and only if they are $\cF$-conjugate.
\end{proof}

\begin{proposition}\label{pro:Main}
Let $G$ be a discrete group realizing the saturated fusion system $\cF$, and $M$ be any $\bbZ _{(p)} \cO ^c (\cF)$-module.
Suppose that $\cC$ is the collection of all $\cF$-centric subgroups in $S$, and $\cC'$ be the collection of $p$-centric subgroups in $G$.
Then for each $i\geq 0$, the higher limit $\lim^i _{\cO ^c (\cF)} M$  is isomorphic to the $i$-th cohomology of the sequence 
\begin{align*} 
\quad 0 \to \bigoplus _{[P_0] \in |\cC|_0/\cF} M (P_0) ^{\Aut _\cF(P_0)} \to \bigoplus _{[P_0< P_1] \in |\cC| _1/\cF} M (P_0) ^{\Aut_\cF(P_0< P_1)} \to \cdots \to 0
\end{align*}
where homomorphisms in the sequence are induced by the coboundary maps in the cochain complex $C^* _G (|\cC'|; \cM)$.
\end{proposition}

\begin{proof}
Let $G$ be a discrete group with Sylow $p$-subgroup $S$ realizing $\cF$, and $\cC'$ be the collection of all $p$-centric subgroups of $G$. Let $\cM$ be the coefficient system on $|\cC '|$ obtained by precomposing $M$ with $\overline \xi: \Delta (|\cC'| )_G \to \cO ^c (\cF)^{op}$. For each $n\geq 0$, we have $$C_G^n (|\cC'|; \cM ) = \Bigl ( \prod _{(P_0< \cdots < P_n ) \in |\cC'| _n } M (P_0 ) \Bigr ) ^G \cong \bigoplus _{[P_0<\cdots < P_n ] \in |\cC'|_n /G } M(P_0) ^{N_G (P_0<\cdots < P_n )}.$$
%where $N_G (P_0 < \cdots < P_n)=\bigcap_{i=0}^n N_G(P_i)$. 
%Hence the cochain complex $C^n_G (|\cC'| ; \cM)$ is a sequence of the form
%\begin{align*} 
%\quad 0 \to \bigoplus _{[P] \in |\cC'|_0/G} M (P) ^{N_G (P)} \to \bigoplus _{[P_0< P_1] \in |\cC'| _1/G} M (P_0) ^{N_G(P_0< P_1)} \to \cdots \to 0.
%\end{align*}
By Lemma \ref{lem:IsomQuotients}, for each $n\geq 0$,  the set of $G$-orbits $|\cC'|_n/G$ is isomorphic to  $|\cC|_n/\cF$.  Since $M$ is a $\bbZ _{(p)} \cO ^c (\cF)$-module, 
$C_G(P_0)$ acts trivially on $M(P_0)$. Since $C_G (P_n) \le C_G (P_0)$, the centralizer $C_G(\sigma):=\cap _{i=0} ^n C_G (P_i)=C_G(P_n)$ also acts trivially on $M(P_0)$. Hence for each chain $\sigma$ in $|\cC|_n$,
$$M(P_0) ^{N_G (\sigma)}= M(P_0) ^{N_G (\sigma )/ C_G (P_n)}=
M(P_0) ^{\Aut _{\cF} (\sigma )}.$$
So the cochain complex $C^*_G (|\cC'|; \cM)$ is isomorphic to the sequence given in the proposition. By Theorem \ref{thm:Main}, the higher limit $\lim ^i _{\cO ^c (\cF)} M$ is isomorphic to the $i$-th cohomology of this sequence.
\end{proof}

%The proof of Corollary \ref{cor:Main} follows from Proposition \ref{pro:Main}.

\begin{proof}[Proof of Corollary \ref{cor:Main}]
Consider the sequence in Proposition \ref{pro:Main} with $M=H^n (-)$, the group cohomology functor.  Note that $\lim_{\cO ^c (\cF)} H^n (-) \cong H^n (\cF)$, so if we add $H^n (\cF)$ to the left-hand side of this sequence, the sequence will be exact for all $n\geq 0$ if and only if  the subgroup decomposition over $\cF$-centric subgroups is sharp (see Definition \ref{def:BFSharp}). 
\end{proof}
 
%-------------------------

\section{A sequence induced by the normalizer decomposition}\label{sect:NormalizerDec}

In this section we discuss the normalizer decomposition for $p$-local finite groups, introduced by Libman \cite{Libman-NormDec}, and we obtain a fusion system version of the sequence \eqref{eqn:GroupExactSeq-Normalizer} stated in the introduction.  
 
Let $\cF$ be a saturated fusion system over $S$, and let $\cF^c$ denote the full subcategory whose objects are $\cF$-centric subgroups in $S$.
Recall that two strict chains of subgroups $\sigma = (P_0< \cdots < P_n)$ and $\tau = (Q_0 < \cdots < Q_n)$ are  $\cF$-conjugate if there is an $\cF$-isomorphism $\varphi : P_n \to Q_n$ such that $\varphi (P_i)=Q_i$ for all $i$. 
The \emph{subdivision category of $\cF^c$} is the category  $\overline{s}d(\cF^c)$ whose objects are the $\cF$-conjugacy classes $[\sigma]$ of (strict) chains $\sigma :=(P_0<P_1< \cdots <P_n)$ of $\cF$-centric subgroups of $S$,
and whose morphisms are defined so that there is a unique morphism $[\sigma] \to [\tau]$ if $\tau$ is $\cF$-conjugate to a 
subchain of $\sigma$.

Associated to each fusion system, there is a unique centric linking system $\cL$ up to isomorphism whose objects are the $\cF$-centric subgroups of $S$ and whose morphisms are defined by a list of properties (see \cite[Definition 10.9]{Yalcin-Higher}). The triple $(S, \cF, \cL)$ is called a $p$-local finite group. 
%Normalizer decompositions for linking systems were introduced by Libman in \cite{Libman-NormDec}.
There is a quotient functor $\pi : \cL \to \cF^c$ as part of the definition of a centric linking system.

For each inclusion map $\mathrm{inc} _P ^Q : P \hookrightarrow Q$ in $\cF^c$, there is a morphism $\iota _P ^Q$ in $\cL$ such that $\pi (\iota_P^Q)=\mathrm{inc}_P ^Q$. These morphisms are defined such that $\iota _P ^P=\id_P$ for all $P\in \cF^c$ and that $\iota _Q ^R \circ \iota _P ^Q =\iota _P ^R $ whenever $P<Q<R$ in $\cF^c$.
For every chain $\sigma =(P_0< \cdots <P_n)$ of subgroups in $\cF^c$, the automorphism group $\Aut_{\cL}(\sigma)$ is defined to be the subgroup
$$\Aut _{\cL} (\sigma ) =\{ (\varphi_i) \in \prod _i \Aut _{\cL} (P_i) \, |\, \varphi_j \circ \iota ^{P_j}  _{P_i} = \iota _{P_i} ^{P_j} \circ \varphi _i \text{  for all } i\le j \} \subseteq \prod _i \Aut _{\cL} (P_i).$$

If $\tau$ is a face of $\sigma$, then there is an associated group homomorphism $\Res ^{\sigma} _{\tau} : \Aut _{\cL } (\sigma) \to \Aut _{\cL } (\tau)$ defined by restriction of automorphisms. This induces a map $B\Res ^{\sigma} _{\tau} : B\Aut_{\cL} (\sigma) \to B\Aut _{\cL} (\tau)$ on the classifying spaces. In \cite[Theorem A]{Libman-NormDec}, Libman proves that there is a natural weak homotopy equivalence 
\[
\hocolim _{\overline sd (\cF^c)} \delta_{\cF^c} \xrightarrow{\simeq} |\cL|
\]
where $\delta _{\cF^c}: \overline sd (\cF ^c) \to Top$ is a functor such that for every chain $\sigma$ of subgroups in $\cF^c$, there is a natural weak homotopy equivalence $\delta_{\cF^c} ([\sigma]) \simeq B\Aut _{\cL} (\sigma)$. The naturality here means that if $\tau$ is a face of $\sigma$, then there is a commuting square
\[
\xymatrix{B\Aut_{\cL}(\sigma) \ar[d]^{B\Res^{\sigma} _{\tau}} \ar[r]^{\simeq} & \delta_{\cF^c} ([\sigma]) \ar[d] \\
B\Aut_{\cL} (\tau) \ar[r]^{\simeq} & \delta_{\cF^c} ([\tau])}
\]

Composing the functor $\delta _{\cF^c}$ with the cohomology functor $H^n (-)$, we obtain a contravariant functor $$\cH ^n _{A_{\cL}} : \overline{s}d(\cF^c) ^{op} \to \bbF_p\text{-Mod}$$
such that for every simplex $\sigma$ there is an isomorphism 
$\cH ^n _{A_{\cL}}  ([\sigma]) \cong H^n (\Aut _{\cL } (\sigma))$. In \cite{Yalcin-Higher}, the following theorem is proved.

\begin{theorem}[{\cite[Theorem 10.15]{Yalcin-Higher}}]\label{thm:Yalcin} Let $\cF$ be a saturated fusion system over $S$. For every  $n\geq 0$,   there is an isomorphism
$$\underset{\cO^c (\cF) } {\lim {}^*} H^n (-)  \cong \underset{\overline{s}d(\cF^c)}{\lim {}^*} \cH ^n _{A_{\cL}}.$$
\end{theorem}

The normalizer decomposition for a $p$-local finite group is said to be \emph{sharp} if the higher limits $\lim ^i _{\overline{s}d(\cF^c)} \cH ^n _{A_{\cL}}$ vanish for all $i>0$, $n\geq 0$. Theorem \ref{thm:Yalcin} says that the normalizer decomposition for a $p$-local finite group $(S, \cF, \cL)$ is sharp if and only if the subgroup decomposition for $\cF$ is sharp over $\cF$-centric subgroups.

We can reinterpret the higher limits $\lim _{\overline{s}d(\cF^c)} ^* \cH ^n _{A_{\cL}}$ using another category. Let $G$ be a discrete group realizing the fusion system $\cF$, and let $\cC$ be a collection of subgroups in $G$. 
We denote by $sd _{\cC} (G)$ the poset category of chains $\sigma =(P_0<\cdots < P_n)$ of subgroups in $\cC$, where there
is a morphism $\tau \to \sigma$ if $\tau$ is a subchain of $\sigma$.
The opposite category of $sd_{\cC} (G)$ is denoted by $Sd _{\cC} (G)$.

There is a $G$-action on $Sd_{\cC} (G)$ defined by $G$-conjugation. 
The category of \emph{$G$-orbits of chains in $\cC$} is the category $Sd_{\cC} (G)/G$ 
whose objects are the $G$-orbits 
$[\sigma]$ of the chains $\sigma$ in $Sd_{\cC} (G)$ and such that there is a unique morphism 
$[\sigma] \to [\tau]$ in $Sd_{\cC} (G)/G$ if there is an element $g\in G$ such that $\tau $ 
is a subchain of $g \sigma$. 

\begin{lemma}\label{lem:IsomorphicPosets}
Let $G$ be a discrete group realizing the fusion system $\cF$, and $\cC$ denote the collection of all $p$-centric subgroups in $G$. Then the poset category $Sd_{\cC} (G)/G$ is isomorphic to $\overline{s}d(\cF ^c )$. 
\end{lemma}

\begin{proof} This follows from \cite[Lemma 8.9]{Yalcin-Higher} which states that every $p$-subgroup of $G$ is $p$-centric if and only  if it is conjugate to an $\cF$-centric subgroup in $S$. 
\end{proof}

As a consequence of Lemma \ref{lem:IsomorphicPosets}, the isomorphism in Theorem  \ref{thm:Yalcin} can be stated as follows:
\begin{align}\label{eqn:FirstStep}
\underset{\cO^c (\cF) } {\lim {}^*} H^n (-)  \cong \underset{Sd_{\cC} (G)/G}{\lim {}^*} \cH ^n _{A_{\cL}}.
\end{align}
The next step is to find an interpretation of the higher limits on the right as the equivariant cohomology of the complex $|\cC|$ with coefficients in an equivariant coefficient system.

Consider the functor $$\eta :  \Delta (|\cC| )_G \to  sd_{\cC} (G)/G = (Sd_{\cC} (G)/G )^{op}$$
which sends each simplex $\sigma=(P_0< \cdots < P_n)$ to its $G$-orbit $[\sigma]$ in $sd_{\cC} (G)/G$, and each morphism $(g, f) : \tau \to \sigma$ in $\Delta (|\cC| )_G$ to the unique morphism  $[\tau] \to [\sigma]$ in $sd_{\cC} (G)/G$. For a non-strict chain $P_0\le \cdots\le P_n$, the functor $\eta$ first deletes the repeated subgroups, and then sends the resulting strict chain to its orbit. For higher limits over the subdivision category, we have the following result which is due to Slominska \cite{Slominska-SpectralSeq} (see also \cite[Proposition 7.8]{Yalcin-Higher} and \cite[Proposition 7.1]{Grodal-Higher}).

\begin{proposition}\label{pro:Slominska} Let $N : (Sd_{\cC}(G)/G)^{op} \to R$-Mod be an arbitrary functor. Then
$$\underset{Sd_{\cC} (G) /G }{\lim {}^*}  N \cong 
H^* _G (|\cC| ; \cN)$$
where $\cN$  is the equivariant coefficient system defined by precomposing $N$ with the functor $\eta : \Delta (|\cC|) _G \to  (Sd_{\cC} (G)/G)^{op}$.  
\end{proposition}

Combining the isomorphism in \eqref{eqn:FirstStep} with the isomorphism in Proposition \ref{pro:Slominska}, we obtain 
\begin{equation}\label{eqn:isom2}
\underset{\cO^c (\cF) } {\lim {}^*} H^n (-) \cong 
H^* _G ( |\cC|; \cN)
\end{equation}
where $\cN$ is the equivariant coefficient system on $|\cC|$ such that  $$\cN (\sigma)\cong H^n (\Aut _{\cL} (\sigma))$$
for every fully $\cF$-normalized chain $\sigma$ in $|\cC|$. Using the definition of the cochain complex for the equivariant cohomology, we conclude the following: 

\begin{proposition}\label{pro:MainAut}
Let $G$ be a discrete group realizing the saturated fusion system $\cF$. Suppose that $\cC$ is the collection of $\cF$-centric subgroups of $S$,  and $\cC'$ is the collection of $p$-centric subgroups in $G$. Consider the sequence obtained by adding $H^n (\cF)$ to the beginning of the cochain complex $C^* _G ( |\cC'|; \cN)$ described above:
$$0 \to H^n (\cF) \to \bigoplus _{[P]\in |\cC|_0 /\cF} H^n (\Aut_{\cL}(P)) \to \bigoplus _{[P_0< P_1] \in |\cC| _1/\cF} H^n ( \Aut_{\cL} (P_0< P_1) )   \to \cdots \to 0. $$
These sequences are exact for all $n\geq 0$ if and only if the subgroup decomposition for $\cF$ is sharp over the collection of $\cF$-centric subgroups in $S$.
\end{proposition}

\begin{proof} This follows from the isomorphism in \eqref{eqn:isom2}.   
\end{proof}

We conclude the paper with the proof of Corollary \ref{cor:MainNorm}.

\begin{proof}[Proof of Corollary \ref{cor:MainNorm}] In the proof of \cite[Lemma 10.14]{Yalcin-Higher}, we showed that for every fully normalized chain $\sigma=(P_0<\cdots< P_n)$ of $\cF$-centric subgroups in $|\cC|$, there is an isomorphism
$$ H^* ( N_{\cF} (\sigma) ) \cong H^* (\Aut _{\cL} (\sigma ))$$
where $N_{\cF} (\sigma)$ denotes the normalizer fusion system of $\sigma$ defined in Definition \ref{def:NormFusion}. By  choosing fully $\cF$-normalized representatives for $|\cC|_n/\cF$, and using these cohomology isomorphisms, we see that the cochain complex in Proposition \ref{pro:MainAut} is isomorphic to the cochain complex given in Corollary \ref{cor:MainNorm}.
\end{proof}

%-------------------------

\end{document}